\documentclass[10pt,reqno,oneside]{amsart}
\usepackage{graphicx,verbatim}
\usepackage{amssymb,cite}
\usepackage{epstopdf}
\usepackage{graphicx}
\usepackage{amsmath}
\usepackage{amsthm, enumerate}
\usepackage{mathrsfs}
\usepackage{caption}
\usepackage{subcaption}
\usepackage[normalem]{ulem}   
\allowdisplaybreaks

\chardef\forshowkeys=0
\chardef\showllabel=0
\chardef\refcheck=0
\chardef\sketches=0
\chardef\showfont=0         

\usepackage{float}

\floatstyle{boxed}

\usepackage{marginnote}

\usepackage[colorlinks=true, pdfstartview=FitV, linkcolor=blue, citecolor=blue, urlcolor=blue]{hyperref}

\ifnum\forshowkeys=1
  
  \usepackage[notref,notcite,color]{showkeys}
\fi

\author[M.S.~Ayd{\i}n]{Mustafa Sencer Ayd{\i}n}
\address{Department of Mathematics, University of Southern California, Los Angeles, CA 90089}
\email{maydin@usc.edu}

\title[The inviscid limit for the Navier-Stokes equations with Navier boundary conditions]{Uniform bounds and the inviscid limit for the Navier-Stokes equations with Navier boundary conditions}
\author[I.~Kukavica]{Igor Kukavica}
\address{Department of Mathematics\\
 University of Southern California\\
 Los Angeles, CA 90089}
\email{kukavica@usc.edu}

\usepackage{enumitem}
\usepackage{datetime}
\usepackage{fancyhdr}
\usepackage{comment}
\allowdisplaybreaks
\usepackage[margin=1in]{geometry}
\usepackage{amsmath, amsthm, amssymb}
\usepackage{times}
\usepackage{graphicx}
\usepackage[usenames,dvipsnames,svgnames,table]{xcolor}

\ifnum\showfont=1
  \usepackage{fontspec-xetex}
\fi

\ifnum\refcheck=1
\usepackage{refcheck}
\fi

\begin{document}
 
\def\inprogress{{\colg IN PROGRESS} }
\def\bnew{\colr {}}
\def\enew{\colb {}}
\def\bold{\colu {}}
\def\eold{\colb{}}
\def\YY{X}
\def\OO{\mathcal O}
\def\SS{\mathbb S}
\def\CC{\mathbb C}
\def\RR{\mathbb R}
\def\TT{\mathbb T}
\def\ZZ{\mathbb Z}
\def\HH{\mathbb H}
\def\RSZ{\mathcal R}
\def\LL{\mathcal L}
\def\SL{\LL^1}
\def\ZL{\LL^\infty}
\def\GG{\mathcal G}
\def\tt{\langle t\rangle}
\def\erf{\mathrm{Erf}}
\def\mgt#1{\textcolor{magenta}{#1}}
\def\ff{\rho}
\def\gg{G}
\def\sqrtnu{\sqrt{\nu}}
\def\ww{w}
\def\ft#1{#1_\xi}
\def\ges{\gtrsim}
\renewcommand*{\Re}{\ensuremath{\mathrm{{\mathbb R}e\,}}}
\renewcommand*{\Im}{\ensuremath{\mathrm{{\mathbb I}m\,}}}

\ifnum\showllabel=1
\def\llabel#1{\marginnote{\color{gray}\rm(#1)}[-0.0cm]\notag}
\else
\def\llabel#1{\notag}
\fi

\newcommand{\norm}[1]{\left\|#1\right\|}
\newcommand{\nnorm}[1]{\lVert #1\rVert}
\newcommand{\abs}[1]{\left|#1\right|}
\newcommand{\NORM}[1]{|\!|\!| #1|\!|\!|}
\theoremstyle{plain}
\newtheorem{theorem}{Theorem}[section]
\newtheorem{Theorem}{Theorem}[section]
\newtheorem{corollary}[theorem]{Corollary}
\newtheorem{Corollary}[theorem]{Corollary}
\newtheorem{proposition}[theorem]{Proposition}
\newtheorem{Proposition}[theorem]{Proposition}
\newtheorem{Lemma}[theorem]{Lemma}
\newtheorem{lemma}[theorem]{Lemma}
\theoremstyle{definition}
\newtheorem{definition}{Definition}[section]
\newtheorem{remark}[Theorem]{Remark}
\def\theequation{\thesection.\arabic{equation}}
\numberwithin{equation}{section}
\definecolor{mygray}{rgb}{.6,.6,.6}
\definecolor{myblue}{rgb}{9, 0, 1}
\definecolor{colorforkeys}{rgb}{1.0,0.0,0.0}
\newlength\mytemplen
\newsavebox\mytempbox
\makeatletter
\newcommand\mybluebox{%
  \@ifnextchar[
  {\@mybluebox}%
  {\@mybluebox[0pt]}}
\def\@mybluebox[#1]{%
  \@ifnextchar[
  {\@@mybluebox[#1]}%
  {\@@mybluebox[#1][0pt]}}
\def\@@mybluebox[#1][#2]#3{
  \sbox\mytempbox{#3}%
  \mytemplen\ht\mytempbox
  \advance\mytemplen #1\relax
  \ht\mytempbox\mytemplen
  \mytemplen\dp\mytempbox
  \advance\mytemplen #2\relax
  \dp\mytempbox\mytemplen
  \colorbox{myblue}{\hspace{1em}\usebox{\mytempbox}\hspace{1em}}}
 \makeatother
\def\XX{{\mathcal X}}
\def\XXT{{\mathcal X}_T}
\def\XXTzero{{\mathcal X}_{T_0}}
\def\XXi{{\mathcal X}_\infty}
\def\YY{{\mathcal Y}}
\def\YYT{{\mathcal Y}_T}
\def\YYTzero{{\mathcal Y}_{T_0}}
\def\YYi{{\mathcal Y}_\infty}
\def\cc{\text{c}}
\def\rr{r}
\def\weaks{\text{\,\,\,\,\,\,weakly-* in }}
\def\inn{\text{\,\,\,\,\,\,in }}
\def\cof{\mathop{\rm cof\,}\nolimits}
\def\Dn{\frac{\partial}{\partial N}}
\def\Dnn#1{\frac{\partial #1}{\partial N}}
\def\tdb{\tilde{b}}
\def\tda{b}
\def\qqq{u}
\def\lat{\Delta_2}
\def\biglinem{\vskip0.5truecm\par==========================\par\vskip0.5truecm}
\def\inon#1{\hbox{\ \ \ \ \ \ \ }\hbox{#1}}                
\def\onon#1{\inon{on~$#1$}}
\def\inin#1{\inon{in~$#1$}}
\def\FF{F}
\def\andand{\text{\indeq and\indeq}}
\def\ww{w(y)}
\def\ll{{\color{red}\ell}}
\def\ee{\epsilon_0}
\def\startnewsection#1#2{ \section{#1}\label{#2}\setcounter{equation}{0}}   
\def\loc{\text{loc}}
\def\nnewpage{ }
\def\sgn{\mathop{\rm sgn\,}\nolimits}    
\def\Tr{\mathop{\rm Tr}\nolimits}    
\def\div{\mathop{\rm div}\nolimits}
\def\curl{\mathop{\rm curl}\nolimits}
\def\dist{\mathop{\rm dist}\nolimits}  
\def\supp{\mathop{\rm supp}\nolimits}
\def\indeq{\quad{}}           
\def\period{.}                       
\def\semicolon{\,;}                  
\def\colr{\color{red}}
\def\colrr{\color{black}}
\def\colb{\color{black}}
\def\coly{\color{lightgray}}
\definecolor{colorgggg}{rgb}{0.1,0.5,0.3}
\definecolor{colorllll}{rgb}{0.0,0.7,0.0}
\definecolor{colorhhhh}{rgb}{0.3,0.75,0.4}
\definecolor{colorpppp}{rgb}{0.7,0.0,0.2}
\definecolor{coloroooo}{rgb}{0.45,0.0,0.0}
\definecolor{colorqqqq}{rgb}{0.1,0.7,0}
\def\colg{\color{colorgggg}}
\def\collg{\color{colorllll}}
\def\cole{\color{colorpppp}}
\def\cole{\color{coloroooo}}
\def\coleo{\color{colorpppp}}
\def\cole{\color{black}}
\def\colu{\color{blue}}
\def\colc{\color{colorhhhh}}
\def\colW{\colb}   
\definecolor{coloraaaa}{rgb}{0.6,0.6,0.6}
\def\colw{\color{coloraaaa}}
\def\comma{ {\rm ,\qquad{}} }            
\def\commaone{ {\rm ,\quad{}} }          
\def\lec{\lesssim}
\def\nts#1{{\color{red}\hbox{\bf ~#1~}}} 
\def\ntsf#1{\footnote{\color{colorgggg}\hbox{#1}}}
\def\ntsik#1{{\color{purple}\hbox{\bf ~#1~}}} 
\def\blackdot{{\color{red}{\hskip-.0truecm\rule[-1mm]{4mm}{4mm}\hskip.2truecm}}\hskip-.3truecm}
\def\bluedot{{\color{blue}{\hskip-.0truecm\rule[-1mm]{4mm}{4mm}\hskip.2truecm}}\hskip-.3truecm}
\def\purpledot{{\color{colorpppp}{\hskip-.0truecm\rule[-1mm]{4mm}{4mm}\hskip.2truecm}}\hskip-.3truecm}
\def\greendot{{\color{colorgggg}{\hskip-.0truecm\rule[-1mm]{4mm}{4mm}\hskip.2truecm}}\hskip-.3truecm}
\def\cyandot{{\color{cyan}{\hskip-.0truecm\rule[-1mm]{4mm}{4mm}\hskip.2truecm}}\hskip-.3truecm}
\def\reddot{{\color{red}{\hskip-.0truecm\rule[-1mm]{4mm}{4mm}\hskip.2truecm}}\hskip-.3truecm}
\def\gdot{\greendot}
\def\tdot{\gdot}
\def\bdot{\bluedot}
\def\ydot{\cyandot}
\def\rdot{\reddot}
\def\fractext#1#2{{#1}/{#2}}
\def\ii{\hat\imath}
\def\fei#1{\textcolor{blue}{#1}}
\def\vlad#1{\textcolor{cyan}{#1}}
\def\igor#1{\text{{\textcolor{colorqqqq}{#1}}}}
\def\igorf#1{\footnote{\text{{\textcolor{colorqqqq}{#1}}}}}
\newcommand{\p}{\partial}
\newcommand{\UE}{U^{\rm E}}
\newcommand{\PE}{P^{\rm E}}
\newcommand{\KP}{K_{\rm P}}
\newcommand{\uNS}{u^{\rm NS}}
\newcommand{\vNS}{v^{\rm NS}}
\newcommand{\pNS}{p^{\rm NS}}
\newcommand{\omegaNS}{\omega^{\rm NS}}
\newcommand{\uE}{u^{\rm E}}
\newcommand{\vE}{v^{\rm E}}
\newcommand{\pE}{p^{\rm E}}
\newcommand{\omegaE}{\omega^{\rm E}}
\newcommand{\ua}{u_{\rm   a}}
\newcommand{\va}{v_{\rm   a}}
\newcommand{\omegaa}{\omega_{\rm   a}}
\newcommand{\ue}{u_{\rm   e}}
\newcommand{\ve}{v_{\rm   e}}
\newcommand{\omegae}{\omega_{\rm e}}
\newcommand{\omegaeic}{\omega_{{\rm e}0}}
\newcommand{\ueic}{u_{{\rm   e}0}}
\newcommand{\veic}{v_{{\rm   e}0}}
\newcommand{\up}{u^{\rm P}}
\newcommand{\vp}{v^{\rm P}}
\newcommand{\tup}{{\tilde u}^{\rm P}}
\newcommand{\bvp}{{\bar v}^{\rm P}}
\newcommand{\omegap}{\omega^{\rm P}}
\newcommand{\tomegap}{\tilde \omega^{\rm P}}
\renewcommand{\up}{u^{\rm P}}
\renewcommand{\vp}{v^{\rm P}}
\renewcommand{\omegap}{\Omega^{\rm P}}
\renewcommand{\tomegap}{\omega^{\rm P}}
\def\hh{\text{h}}
\def\cco{\text{co}}

\begin{abstract}
  We consider the
  vanishing viscosity problem for solutions of the Navier-Stokes equations
  with Navier boundary conditions in the half-space.
  We lower the currently known conormal regularity needed to establish that the inviscid limit holds.
  Our requirement for the Lipschitz initial data is that the first four conormal derivatives are bounded along with two for the gradient.
In addition,
we establish a new class of initial data for the local existence and
 uniqueness for the Euler equations in the half-space or a channel
 for initial data in the conormal space without conormal requirements
 on the gradient.
\end{abstract}

\maketitle

\startnewsection{Introduction}{sec.int}
We address the vanishing viscosity limit problem for the three-dimensional incompressible Navier-Stokes equations
 \begin{align}
  \partial_t u^\nu - \nu \Delta u^\nu + u^\nu \cdot \nabla u^\nu + \nabla p^\nu =0 
  \comma
  \nabla \cdot u^\nu = 0 \comma (x,t) \in \Omega \times (0,T)
  ,
  \label{NSE}
 \end{align}
with the Navier boundary conditions
 \begin{align}
  u^\nu\cdot n = 0 \comma 
  \left(\frac{1}{2}(\nabla u^\nu + \nabla^T u^\nu)\cdot n\right)_\tau
  =-\mu u^\nu_\tau
  \comma (x,t) \in \partial \Omega \times (0,T)
  ,
  \label{navierbdry}
 \end{align}
 where $\mu \in \mathbb{R}$ is a constant, $n$ is the outward unit normal vector, and 
 $v_\tau = v - (v \cdot n)n$
 is the tangential part of~$v$.
 It is a well-known problem whether the solutions $\{u^\nu\}_\nu$ for \eqref{NSE}
 converge to a solution $u$ for the Euler equations
 \begin{align}
  u_t + u\cdot \nabla u + \nabla p =0
  \comma \nabla \cdot u = 0
  \comma (x,t) \in \Omega \times (0,T)
  ,
  \label{euler}
 \end{align}
 with the slip boundary condition
 \begin{align}
  u \cdot n = 0 
  \text{ on } \partial \Omega \times (0,T)
,
  \label{eulerb}
 \end{align}
 as~$\nu \to 0$. 
 For the no-slip boundary condition 
 	$
  u^\nu = 0 
  \text{ on } \partial \Omega \times (0,T),
  $
 the inviscid limit problem 
 is to a large extent open due to the presence of strong boundary layers;
 see \cite{CL+,FTZ,GK+,GMM,K,Ke,KLS,KVW,M,MM,SC1,SC2,TW} to point out only few references.
 However, substantial progress has been made for the Navier-type boundary conditions~\eqref{navierbdry}.
 When $\Omega \subset \mathbb{R}^3$, this problem was first addressed
 by Iftimie and Planas in~\cite{IP}, where they
 consider \eqref{euler}--\eqref{eulerb} 
 with the initial data $u_0 \in H^s$ for $s > 2.5$,
 the condition that guarantees the local-in-time well-posedness of Euler equations.
 Then, by assuming the $L^2$ convergence of the initial 
 data $u_0^\nu$ for \eqref{NSE} to $u_0$, they prove
 the $L^\infty_t L^2_x$ convergence of the weak Leray solutions
 to the local strong solution of the Euler equations.
 A quantitative result is obtained by Gie and Kelliher in~\cite{GK},
 who considered
 the generalized slip condition
 \begin{align}
  u^\nu\cdot n = 0 \comma 
  \left(\frac{1}{2}(\nabla u^\nu + \nabla^T u^\nu)\cdot n\right)_\tau
  = (Au^\nu)_\tau
  \comma (x,t) \in \partial \Omega \times (0,T),
  \label{gnavierbdry}
 \end{align}
 where $A$ is a smooth symmetric tensor.
 Assuming that the initial data belongs to $H^5$,
 the authors establish the 
 $L^\infty_t L^2_x$ and $L^2_t H^1_x$ convergence 
 with the rates $\nu^\frac{3}{4}$ and $\nu^\frac{1}{4}$, respectively.
 These rates were improved in~\cite{XZ2}
 in the case of the generalized vorticity-slip condition
 \begin{align}
  u^\nu\cdot n = 0 \comma 
  (\curl u^\nu) \times n = (B u^\nu)_\tau  
  \comma (x,t) \in \partial \Omega \times (0,T),
  \llabel{gvslip}
 \end{align} 
 which is equivalent to 
 \eqref{gnavierbdry} for some smooth symmetric tensor~$A$.
 Their initial data belongs to $H^3$ and satisfies~\eqref{mrassumption}.
 Under these assumptions, the authors establish, for any $s>0$ and $2\le p < \infty$,
 the convergence in
$L^\infty_t L^2_x, L^\infty_t H^1_x, L^\infty_t W^{1,p}_x, L^\infty_{t,x}  
$,
 with the rates
$\nu^{1-s}, \nu^{\frac{1-s}{2}}, \nu^\frac{1-s}{p}, \nu^{\frac{2}{5}(1-s)}
$, respectively.
 We note that the inviscid limit problem
 has also been studied with several other Navier-type boundary conditions
 such as 
 \begin{equation}
  \label{other}
  u^\nu \cdot n = 0
  \comma
  \begin{cases}
   &(\curl u)\cdot n = 0 \\
   &(\curl u)\cdot n = 0 = \Delta u \times n  \\
   & (\curl u)_\tau = 0 \\
   & (\curl u) \times n = 0 \\
   &    ((-pI + \frac{\nu}{2}(\nabla u^\nu + \nabla^T u^\nu))\cdot n)\cdot \tau = 0
   ;
  \end{cases}
 \end{equation}
we refer the reader to~\cite{BdVC1,BdVC2,BdVC3,BdVC4,BS1,BS2,CQ,DN,IS,NP,TWZ,WXZ,X,XX,XZ1} for the corresponding results.

In~\cite{MR1},
 Masmoudi and Rousset
 derived uniform-in-$\nu$ estimates for the Navier-Stokes system  
 in the functional framework of Sobolev conormal spaces.
 In particular, they assumed that
 \begin{align}
  (u,\nabla u)|_{t=0}  \in H^7_\cco \times (H^6_\cco \cap W^{1,\infty}_\cco)
  ,\label{mrassumption}
 \end{align}
 (see the definitions~\eqref{con.def} and~\eqref{norm} below)
 and established the local well-posedness of Euler equations in the same functional setting.
 Moreover, the unique solution for \eqref{euler}--\eqref{eulerb}
 is realized as the $L^\infty_{x,t} \cap L^\infty_t L^2_x$ limit of $u^\nu$
 as~$\nu \to 0$. After establishing the uniform estimates, 
 they use a compactness argument to prove qualitatively that the inviscid limit holds.
 
 In this work, we lower the regularity assumption~\eqref{mrassumption} in~\cite{MR1}
 for the case of the three-dimensional half-space.
 In particular, we address the initial data 
 \begin{align}
  (u,\nabla u)|_{t=0}  \in (H^4_\cco \cap W^{2,\infty}_{\cco}) \times (H^2_\cco \cap L^\infty)
  \label{ourassumption}
 \end{align}
 and establish the vanishing viscosity limit for $\mu \in \mathbb{R}$; see Theorem~\ref{T02}.
Then, assuming $\mu\geq 0$, we obtain the inviscid limit under the condition
 \begin{align}
  (u,\nabla u)|_{t=0}  \in (H^4_\cco \cap W^{2,\infty}_{\cco}) \times (H^1_\cco \cap L^\infty)
  ;
  \label{ourassumption2}
 \end{align}
see Theorem~\ref{T03}(ii).
 To prove both of these results, we first obtain uniform-in-$\nu$ estimates, stated in Theorems~\ref{T01}
 and~\ref{T03}(i). 
 
 We now briefly describe the key steps allowing us to establish the inviscid limit with
 the data specified in \eqref{ourassumption} and~\eqref{ourassumption2}.
 First, we note that in \eqref{mrassumption} the condition $\nabla u|_{t=0} \in W^{1,\infty}_\cco$
 is required. To maintain this,
 it seems that we need at least six and five conormal derivatives on $u$ and $\nabla u$, respectively, 
 as demonstrated in~\cite{MR2} for the free boundary Navier-Stokes equations.
 Due to the derivative loss in Euler equations, the need to control $\nabla u \in W^{1,\infty}_\cco$ 
 arises from requiring $\nabla u \in H^{m-1}_\cco$ when $u \in H^m_\cco$.
 Namely, in~\cite{MR1}, the authors derive the uniform-in-$\nu$ estimate
 \begin{align}
  \frac{d}{dt}\Vert \nabla u^\nu\Vert_{H^{m-1}_\cco}^2
  \lec
  \Vert \nabla p^\nu\Vert_{H^{m-1}_\cco}\Vert \nabla u^\nu\Vert_{H^{m-1}_\cco}
  +(1+ \Vert \nabla u^\nu\Vert_{W^{1,\infty}_\cco})
  (\Vert \nabla u^\nu\Vert_{H^{m-1}_\cco}^2+\Vert u^\nu\Vert_{H^{m}_\cco}^2)
  ,\llabel{mr1}
 \end{align}
 for all integer $m \ge 7$, and it is challenging to 
 eliminate $\Vert \nabla u^\nu\Vert_{W^{1,\infty}_\cco}$ 
 without introducing more than six conormal derivatives of~$u^\nu$.
 We overcome this in Proposition~\ref{P.Nor} by establishing 
 an estimate that may be simplified to
 \begin{align}
  \frac{d}{dt}\Vert \nabla u^\nu \Vert_{H^m_\cco}^2
  \lec
  \Vert \nabla p^\nu\Vert_{H^m_\cco}\Vert \nabla u^\nu\Vert_{H^m_\cco}
  +  (1+\Vert u^\nu\Vert_{W^{2,\infty}_\cco}+\Vert \nabla u^\nu\Vert_{L^\infty})
  (\Vert \nabla u^\nu\Vert_{H^m_\cco}^2+\Vert u^\nu\Vert_{H^{m+2}_\cco}^2)
  ,
  \llabel{our1}
 \end{align}
 for $m=1,2$, which does not require $\nabla u^\nu \in W^{1,\infty}_\cco$. 
 However, we still need to ensure that $u^\nu \in H^4_\cco$ is bounded uniformly~in~$\nu$,
 which cannot be achieved by repeating the arguments in~\cite{MR1}. This leads to
 \begin{align}
  \frac{d}{dt}\Vert u^\nu\Vert_{H^{m}_\cco}^2
  \lec
  \Vert \nabla p^\nu\Vert_{H^{m-1}_\cco}\Vert u^\nu\Vert_{H^{m}_\cco}
  +(1+ \Vert \nabla u^\nu\Vert_{L^\infty})
  (\Vert \nabla u^\nu\Vert_{H^{m-1}_\cco}^2+\Vert u^\nu\Vert_{H^{m}_\cco}^2)
  ,\label{mr2}
 \end{align}
 for all integer~$m \ge 0$. In our case, \eqref{mr2} dictates $\nabla u^\nu \in H^3_\cco$,
 which is undesirable. For this reason, we revise the conormal estimates and obtain 
 \begin{align}
  \frac{d}{dt}\Vert u(t)\Vert_{4}^2
  \lec
  \Vert \nabla p^\nu\Vert_{H^{3}_\cco}\Vert u^\nu\Vert_{H^{4}_\cco}
  +
  (1+ \Vert \nabla u^\nu\Vert_{L^\infty}+ \Vert u^\nu\Vert_{W^{2,\infty}_\cco})
  \Vert u^\nu\Vert_{H^{4}_\cco}^2
  ;
  \label{our2}
 \end{align}
 for the precise statement, see Proposition~\ref{P.Con}.
 Observe that \eqref{our2} does not contain $\Vert \nabla u^\nu\Vert_{H^m_\cco}$, and
 this suggests that conormal derivatives of $u$ do not explicitly
 depend on conormal derivatives of~$\nabla u$. This
 dependence is only by the term $\Vert u^\nu\Vert_{W^{2,\infty}_\cco}$.
 Similarly, our pressure estimates \eqref{EQ.Pre} also remove the dependence between
 the conormal derivatives of $p$ and of $\nabla u$ present in~\cite{MR1}.
 We remark that all the results are possible by relying on $\Vert u\Vert_{W^{2,\infty}_\cco}$
 rather than $\Vert \nabla u^\nu\Vert_{H^m_\cco}$ or $\Vert \nabla u^\nu\Vert_{W^{1,\infty}_\cco}$.
 The requirement $u_0 \in W^{2,\infty}_\cco$ is much weaker than~\eqref{mrassumption}.
 Indeed, from \eqref{EQ.emb} we can obtain
 \begin{align}
  \Vert u\Vert_{W^{2,\infty}_\cco}^2
  \lec
  \Vert \nabla u\Vert_{H^{2+s_1}_\cco}\Vert u\Vert_{H^{2+s_2}_\cco}
  ,\llabel{hypo}
 \end{align}
 for $s_1+s_2 = 2+$. This means that we can control $\Vert u\Vert_{W^{2,\infty}_\cco}$
 with $\Vert u\Vert_{5}$ and~$\Vert \nabla u\Vert_{2}$. However, controlling 
 $\Vert u\Vert_{W^{2,\infty}_\cco}$
 under the assumptions \eqref{ourassumption} and \eqref{ourassumption2} requires a different approach.
 This is where the cases \eqref{ourassumption} and \eqref{ourassumption2} differ from each other
 depending on the sign of the friction parameter~$\mu$.
 
 After establishing the uniform estimates, we pass to the limit by showing that 
 $\{u^\nu\}$ is a Cauchy sequence in $L^\infty_{t,x} \cap L^\infty_t L^2_x$, from where we obtain
 a unique solution for the Euler equations with an explicit convergence rate
 for the Navier-Stokes solutions given in~\eqref{invlimit}.
 We emphasize that  both of these inviscid limit results
 are new existence results for the Euler equations. Moreover, 
 they are sharper than the previously mentioned results 
 in the sense that our assumptions on the normal derivative of $u$ are minimal. 
 The earlier works on the inviscid limit problem under Navier-type boundary conditions
 either assume that a strong solution
 for the Euler equations exists or that the initial data belongs to a space where 
 the local well-posedness for the Euler equations is already known, such as~$H^{2.5+\delta}$.
 However, here, as can be seen from \eqref{ourassumption} and \eqref{ourassumption2}, 
 we require the boundedness of only a single normal derivative
 at the expense of several conormal derivatives, which is sharp from the
perspective of inviscid limits, but it is further improvable when we disregard 
 the vanishing viscosity approach. 
 
 In Theorem~\ref{T04},
 we obtain a well-posedness result for the Euler equations under the assumption
 \begin{align}
  (u,\nabla u)|_{t=0}  \in (H^4_\cco \cap W^{2,\infty}_{\cco}) \times L^\infty,
  \label{ourassumption3}
 \end{align}
 which completely removes the differentiability requirement
 on $\nabla u$ in the conormal sense. The a~priori estimates in this case
 follow from repeating the proof of Proposition~\ref{P.Ap} for~$\nu = 0$.
However, the inviscid limit approach does not apply as a construction method. This is 
 due to the Laplacian term when we perform $W^{2,\infty}_\cco$
 estimates. To remedy this, we may employ Theorem~\ref{T02} or Theorem~\ref{T03} 
 to obtain an approximate sequence of solutions with a lifespan independent of the approximation. 
 We elaborate more on this in Section~\ref{sec.e}.  
 
In addition to our main results, our uniform estimates given 
for example in~\eqref{ap1}--\eqref{ap4}
are also sharp in integer-valued Sobolev conormal spaces.
In particular, when we consider \eqref{ourassumption}, \eqref{ourassumption2},
or \eqref{ourassumption3} with fewer
integer conormal derivatives on either $u$ or $\nabla u$, we lose control
of the key quantity $\Vert u\Vert_{W^{2,\infty}_\cco}$.
We elaborate more on this in Section~\ref{sec.inf}.
We also emphasize that, in the view of \eqref{eta}, \eqref{EQ62}, and \eqref{EQ07}, 
the assumption $\nabla u_0 \in L^\infty$
in \eqref{ourassumption}, \eqref{ourassumption2}, and \eqref{ourassumption3}
can be replaced by $(\omega_1,\omega_2) \in L^\infty$.
In other words, our results hold when we assume that the tangential part of the initial
vorticity is bounded, and we may remove the requirement that the initial velocity is Lipschitz.
 
 \startnewsection{Preliminaries and the Main Results}{sec.pre}
 
We denote $\Omega =\mathbb{R}^3_+$, and 
 for $x \in \Omega$ we write 
 $x = (x_\hh,z) = (x_1,x_2,z) \in \mathbb{R}^2 \times \mathbb{R}_+$.
We consider
 \begin{align}
  \partial_t u^\nu - \nu \Delta u^\nu + u^\nu \cdot \nabla u^\nu + \nabla p^\nu =0 
  \comma
  \nabla \cdot u^\nu = 0 \comma (x,t) \in \Omega \times (0,T)
  ,
  \label{NSE0}
 \end{align}
with the boundary conditions 
 \begin{align}
  u^\nu_3 = 0 \comma \partial_z u^\nu_\hh = 2\mu u^\nu_\hh \comma (x,t)\in \{z=0\}\times (0,T),
  \label{hnavierbdry}
 \end{align}
where $u_\hh = (u_1,u_2)$.
To introduce Sobolev conormal spaces, let $\varphi(z) = z/(1+z)$, and
 denote 
 $Z_1=\partial_1$, $Z_2=\partial_2$, and $Z_3=\varphi \partial_z$.
 Then define
 \begin{align}
  \begin{split}
   H^m_\cco(\Omega)
   =&
   \{f \in L^2(\Omega) : Z^\alpha f \in L^2(\Omega), \alpha \in \mathbb{N}_0^3, 0\le |\alpha|\le m  \}
   \\
   W^{m,\infty}_\cco(\Omega)
   =&
   \{f \in L^\infty(\Omega) : Z^\alpha f \in L^\infty(\Omega), \alpha \in \mathbb{N}_0^3, 0\le |\alpha|\le m  \}
   .
   \label{con.def}
  \end{split}
 \end{align} 
 These are Hilbert and Banach spaces, respectively, 
 endowed with the norms
 \begin{align}
  \begin{split}
   \Vert f\Vert_{H^m_\cco(\Omega)}^2
   =&
   \Vert f\Vert_{m}^2
   =
   \sum_{|\alpha|\le m} \Vert Z^\alpha f\Vert_{L^2(\Omega)}^2
   \\
   \Vert f\Vert_{W^{m,\infty}_\cco(\Omega)}
   =&
   \Vert f\Vert_{m,\infty}
   =
   \sum_{|\alpha|\le m} \Vert Z^\alpha f\Vert_{L^\infty(\Omega)}
   .
   \label{norm}
  \end{split}
 \end{align}
 When $m=0$, we write $\Vert f\Vert_{L^2}$ and $\Vert f\Vert_{L^\infty}$
 for $\Vert f\Vert_{0}$ and $\Vert f\Vert_{0,\infty}$.
 We are now in a position to state our main results.
 
 \cole
 \begin{Theorem}[Existence and Uniqueness]
  \label{T01}
  Assume that $\nu \in (0,1]$ and $\mu \in \mathbb{R}$.
  Let $u_0 \in H^4_\cco(\Omega) \cap W^{2,\infty}_\cco(\Omega)$,
  with $\nabla u_0 \in H^2_\cco(\Omega) \times L^\infty(\Omega)$, be
  such that $\div u_0 = 0$, and $u_0 \cdot n= 0$ on~$\partial \Omega$.
  Then, for some $T>0$, there exists a unique solution
  $u^\nu \in C([0,T];H^2_\cco(\Omega)) \cap L^\infty(0,T;H^4_\cco(\Omega)\cap W^{2,\infty}_\cco(\Omega))$
  to
  \eqref{NSE0}--\eqref{hnavierbdry} on~$[0,T]$, and
  $M>0$ independent of~$\nu$ such that
  \begin{align}
   \sup_{[0,T]}
   (\Vert u^\nu(t)\Vert_{4}^2
   +\Vert u^\nu(t)\Vert_{2,\infty}^2
   +\Vert \nabla u^\nu(t)\Vert_{2}^2
   +\Vert \nabla u^\nu(t)\Vert_{L^\infty}^2)
   +\nu
   \int_0^T (\Vert \nabla u^\nu(s)\Vert_{5}^2
   +\Vert D^2 u^\nu(s)\Vert_{2}^2)\,ds
   \le M,
   \label{EQ.main1}
  \end{align}
where $M>0$ depends on the norms of the initial data. 
 \end{Theorem}
 \colb
 
As pointed out in the introduction, the requirement
$\nabla u_0 \in H^2_\cco(\Omega) \times L^\infty(\Omega)$
can be replaced by 
$\nabla \omega_0 \in H^2_\cco(\Omega) \times L^\infty(\Omega)$
The analogous remark also applies to Theorems~\ref{T02}, \ref{T03},
and~\ref{T04}.
Upon establishing the existence of $u^\nu$
on a time interval $[0,T]$ which does not depend on $\nu$,
the estimate \eqref{EQ.main1} is sufficient to pass to the limit in \eqref{NSE0}
to obtain~\eqref{euler}. Namely, we have the following inviscid limit statement.
 
 \cole
 \begin{Theorem}[Inviscid Limit]
  \label{T02}
Assume that $\mu \in \mathbb{R}$. Let $u_0$ and $u^\nu$ be as in Theorem~\ref{T01}.
Then, there exists  a unique solution 
$u \in L^\infty(0,T;H^4_\cco(\Omega)\cap W^{2,\infty}_\cco(\Omega))$
with $\nabla u \in L^\infty(0,T;H^2_\cco(\Omega)\cap L^\infty(\Omega))$
to the Euler equations   \eqref{euler}
and $M>0$ independent of~$\nu$ such that
  \begin{align}
   \sup_{[0,T]}
   (\Vert u^\nu -u\Vert_{L^2}^2
   +\Vert u^\nu -u\Vert_{L^\infty}^2)
   \le M\nu^\frac{2}{5},
   \label{invlimit}
  \end{align}
  where $\nu \in [0,1]$ and $M>0$ depends on the norms of the initial data.
 \end{Theorem}
 \colb
 
 When the friction parameter $\mu$ is non-negative, we obtain a
 sharper result by imposing conditions on one less conormal derivative on~$\nabla u$.
 
 \cole
 \begin{Theorem}[A sharper result for non-negative friction]
  \label{T03}
  Assume that $\nu \in (0,1]$ and $\mu \ge 0$.
  Let $u_0 \in H^4_\cco(\Omega) \cap W^{2,\infty}_\cco(\Omega)$,
  and $\nabla u_0 \in H^1_\cco(\Omega) \times L^\infty(\Omega)$ be
  such that $\div u_0 = 0$ and $u_0 \cdot n= 0$ on~$\partial \Omega$.
  \begin{itemize}
   \item[i.] (Existence and Uniqueness)
   For all $\nu \in (0,1]$
   and $\mu \ge 0$ fixed, there exists a unique solution
   $u^\nu \in C([0,T];H^1_\cco(\Omega)) \cap L^\infty(0,T;H^4_\cco(\Omega)\cap W^{2,\infty}_\cco(\Omega))$
   to \eqref{NSE0}--\eqref{hnavierbdry} on~$[0,T]$, for some $T>0$.
 Moreover, there exists $M>0$ independent of $\nu$ such that
   \begin{align}
    \begin{split}
    \sup_{[0,T]}&
    (\Vert u^\nu(t)\Vert_{4}^2
    +\Vert u^\nu(t)\Vert_{2,\infty}^2
    +\Vert \nabla u^\nu(t)\Vert_{1}^2
    +\Vert \nabla u^\nu(t)\Vert_{L^\infty}^2)
    \\&
    +\nu
    \int_0^T (\Vert \nabla u^\nu(s)\Vert_{4}^2
    +\Vert D^2 u^\nu(s)\Vert_{1}^2)\,ds
    \le M,
    \llabel{EQ.main2}
    \end{split}
   \end{align}
where $M$ depends on the norms of the initial data.
 \item[ii.] (Inviscid Limit)
   There exists  a unique solution 
   $u \in L^\infty(0,T;H^4_\cco(\Omega)\cap W^{2,\infty}_\cco(\Omega))$
   with 
   $\nabla u \in L^\infty(0,T;H^1_\cco(\Omega)\cap L^\infty(\Omega))$
   to the Euler equations
   \eqref{euler}
   such that 
   \eqref{invlimit} holds for some constant
   $M>0$ depending on the norms of the initial data
   and independent of~$\nu$.
\end{itemize}
 \end{Theorem}
 \colb
 
 Furthermore, it is possible to improve on the Euler equations part of Theorem~\ref{T03}.
 
 \cole
 \begin{Theorem}[Euler equations in Sobolev conormal spaces]
  \label{T04}
  Let $u_0 \in H^4_\cco(\Omega) \cap W^{2,\infty}_\cco(\Omega)$,
  and $\nabla u_0 \in L^\infty(\Omega)$ be
  such that $\div u_0 = 0$, and $u_0 \cdot n= 0$ on~$\partial \Omega$.
For some $T>0$,
   there exists  a unique solution 
   $u \in L^\infty(0,T;H^4_\cco(\Omega)\cap W^{2,\infty}_\cco(\Omega))$
   with $\nabla u \in L^\infty(\Omega \times (0,T))$
   to the Euler equations
   \eqref{euler}
   and
   $M>0$ independent of~$\nu$ 
   such that
   \begin{align}
    \sup_{[0,T]}
    (\Vert u(t)\Vert_{4}^2
    +\Vert u(t)\Vert_{2,\infty}^2
    +\Vert \nabla u(t)\Vert_{L^\infty}^2)
    \le M,
    \label{EQ.main3}
   \end{align}
where $M>0$ depends on the norms of the initial data. 
 \end{Theorem}
 \colb

 When we perform integration by parts in the term involving the Laplacian, 
 \begin{align}
  -\nu\int_{\Omega} \Delta Z^\alpha u^\nu Z^\alpha u^\nu 
  = 
  \nu\int_{\Omega} |\nabla Z^\alpha u^\nu|^2  +2\nu \mu \int_{\partial \Omega}  |Z^\alpha u^\nu_\hh|^2,
  \llabel{example}
 \end{align}
 we obtain a boundary term that is coercive when~$\mu \ge 0$. 
 When estimating $\Vert u^\nu\Vert_{2,\infty}$ with $\mu \in \mathbb{R}$ a boundary term similar to the one above
 is only controllable when $\nabla u^\nu \in H^2_\cco(\Omega)$; see~\eqref{EQ102}--\eqref{EQ145} for more details. On the other hand,  
 when $\mu \ge 0$, we can estimate $\Vert u\Vert_{2,\infty}$ using~$\Vert \nabla u\Vert_{1}$. This is why in Theorem~\ref{T03}
 we only consider non-negative~$\mu$. We also remark that the time of existence obtained in Theorems~\ref{T01},~\ref{T03}(i), and~\ref{T04} need not be the same. However, for convenience, we denote the time by~$T$.
 
 Note that one of the difficulties is that, unlike
 $Z_1$ and $Z_2$, the operator $Z_3$ does not commute with~$\partial_z$.
 Therefore, we analyze 
 the conormal derivatives with or without
 $Z_3$ separately, and when we do, we write
 \begin{align}
    Z^\alpha = Z^{\tilde{\alpha}}_\hh Z^k_3
    \comma \alpha = (\tilde{\alpha}, k) \in \mathbb{N}_0^2 \times \mathbb{N}_0
    .
    \label{ZhZ3}  
 \end{align} 
 We shall use the following commutator identities.
 
 \cole
 \begin{lemma}
  \label{L01}
  Let $f$ be a smooth function. Then there exist smooth, bounded 
  $\{c^k_{j,\varphi}\}_{j=0}^k$ and $\{\tilde{c}^k_{j,\varphi}\}_{j=0}^k$ of $z$, for $k \in \mathbb{N}$, depending on $\varphi$ such that
  \begin{align}
   \begin{split}    
    &(i)\, Z^k_3 \partial_z f 
    =
    \sum_{j=0}^{k}
    c^k_{j,\varphi} \partial_z Z^j_3 f
    =\partial_z Z_3^kf + \sum_{j=0}^{k-1}
    c^k_{j,\varphi} \partial_z Z^j_3 f
    ,
    \\
    &(ii)\, \partial_z Z^k_3 f 
    =
    \sum_{j=0}^{k}
    \tilde{c}^k_{j,\varphi} Z^j_3 \partial_z f
    =Z_3^k \partial_z f +\sum_{j=0}^{k-1}
    \tilde{c}^k_{j,\varphi} Z^j_3 \partial_z f
    ,
    \\
    &(iii)\, Z^k_3 \partial_{zz} f 
    =
    \sum_{j=0}^{k}
    \sum_{l=0}^{j}
    \left(
    c^j_{l,\varphi} c^k_{j,\varphi} \partial_{zz} Z^l_3 f 
    +
    (c^j_{l,\varphi})' c^k_{j,\varphi} \partial_{z} Z^l_3 f 
    \right)
    ,
    \\
    &(iv)\, \partial_{zz} Z^k_3 f 
    =
    \sum_{j=0}^{k}
    \sum_{l=0}^{j}
    \tilde{c}^j_{l,\varphi} \tilde{c}^k_{j,\varphi} Z^l_3 \partial_{zz} f 
    +
    \sum_{j=0}^{k}
    (\tilde{c}^k_{j,\varphi})' Z^j_3 \partial_z f 
    ,
    \label{EQL02}
   \end{split}
  \end{align}
  where $\tilde{c}^k_{k,\varphi} = 1 = c^k_{k,\varphi}$, and
  the prime indicates the derivative with respect to the variable~$z$.
 \end{lemma}
 \colb
 The proof is by induction and is based on
 \begin{align}
  Z_3 \partial_z f  
  =
  \partial_z Z_3 f - \varphi' \partial_z f 
  .
  \llabel{EQL03}
 \end{align} 
 We also take advantage of the well-known product estimate
 given in the next lemma.
 \cole
 \begin{Lemma}[An anisotropic Sobolev product estimate]
  \label{L02}
  Let $f$ and $g$ be sufficiently smooth functions and $p_i,q_i \in [2,\infty]$, for $i=1,2$. Then
  \begin{align}
   \Vert fg\Vert_{L^2(\Omega)}
   \lec
   \Vert f\Vert_{L^{p_1}_\hh(\mathbb{R}^2;L_z^{q_1}(\mathbb{R}_+))}
   \Vert g\Vert_{L^{p_2}_\hh(\mathbb{R}^2;L_z^{q_2}(\mathbb{R}_+))}
   \lec
   \Vert f\Vert_{H^{r_1}_\hh(\mathbb{R}^2;H_z^{s_1}(\mathbb{R}_+))}
   \Vert g\Vert_{H^{r_2}_\hh(\mathbb{R}^2;H_z^{s_2}(\mathbb{R}_+))}
   \label{EQ.ani}
   ,
  \end{align}
  for any $H^{r_i}_\hh(\mathbb{R}^2) \subset L^{p_i}_\hh(\mathbb{R}^2)$
  and $H^{s_i}_z(\mathbb{R}_+) \subset L^{q_i}_z(\mathbb{R}_+)$, where $i=1,2$ and
  $\frac{1}{p_1} + \frac{1}{p_2} = \frac{1}{q_1} + \frac{1}{q_2} = \frac{1}{2}$. 
 \end{Lemma}
 \colb
 
 We shall also use the inequalities
 \begin{align}
  \Vert Z^{\alpha_1}f 
  Z^{\alpha_2}g\Vert_{L^2}
  \lec
  \Vert f\Vert_{L^\infty}
  \Vert g\Vert_{k}
  +
  \Vert f\Vert_{k}
  \Vert g\Vert_{L^\infty}
  \comma |\alpha_1|+|\alpha_2|=k       
  \comma f,g \in L^\infty \cap H^k_\cco
  \label{EQ.int}
 \end{align}
 and
 \begin{align}
  \Vert f\Vert_{L^\infty}^2
  \lec
  \Vert (I-\Delta_h)^\frac{s_1}{2}\partial_z f\Vert_{L^2}
  \Vert (I-\Delta_h)^\frac{s_2}{2}f\Vert_{L^2}
  \comma s_1,s_2 \ge 0 \comma s_1+s_2 >2.
  \label{EQ.emb}
 \end{align}
 For the proof of \eqref{EQ.int}, see~\cite{G}, while 
 \eqref{EQ.emb} is proven in~\cite{MR2}.
 Next, using the slip boundary condition, the Hardy inequality and the 
 divergence-free condition, we obtain
 \begin{align}
  \left\Vert \frac{u^\nu_3}{\varphi}\right\Vert_{W^{k,p}_\cco}
  \lec \Vert Z_\hh u^\nu_\hh\Vert_{W^{k,p}_\cco}
  ,
  \label{EQ.u3}
 \end{align}
 for all $p \in [1,\infty]$, with $W^{k,p}_\cco$ defined
 analogously to~\eqref{con.def}.
 An additional way that we utilize the divergence-free condition is
by means of Lemma~\ref{L02}.
 Since for any conormal derivative
 $Z^\alpha$ we have
 \begin{align}
  \Vert Z^\alpha \partial_z u^\nu_3\Vert_{L^2}
  \lec
  \Vert u^\nu\Vert_{|\alpha|+1}
  \comma \alpha \in \mathbb{N}^3_0
  ,\label{EQ.u3div}
 \end{align}
 the expression $\Vert Z^{\alpha} u^\nu_3 \Vert_{L^4_\hh L^\infty_z}$
 may be estimated by writing  
 \begin{align}
  \begin{split}
   \Vert Z^{\alpha} u^\nu_3 \Vert_{L^4_\hh L^\infty_z}
   &\lec
   \Vert Z^{\alpha} u^\nu_3 \Vert_{H^\frac{1}{2}_\hh H_z^{\frac{1}{2}+}}
   \lec
   \Vert Z_h Z^{\alpha} u^\nu_3 \Vert_{L^2}
   +\Vert Z^{\alpha} u^\nu_3 \Vert_{L^2}
   +\Vert \partial_z Z^{\alpha} u^\nu_3 \Vert_{L^2}
   +\Vert \partial_z Z_h Z^{\alpha} u^\nu_3 \Vert_{L^2}
   \\&
   \lec
   \Vert u^\nu\Vert_{|\alpha|+2}. 
   \label{EQ.u3div2}
  \end{split}
 \end{align}
 The same reasoning gives
 \begin{align}
  \Vert Z^{\alpha} f \Vert_{L^4_\hh L^\infty_z}
  \lec
  \Vert f\Vert_{|\alpha|+1}+\Vert \partial_z f\Vert_{|\alpha|+1}
  ,\label{EQ15}  
 \end{align}
 for every $f \in H_\cco^{|\alpha|+1}(\Omega)$ such that
 $\partial_z f \in H_\cco^{|\alpha|+1}(\Omega)$.
 
 Now, we fix $\nu \in (0,1]$,
 and denote by $(u,p)$, instead of $(u^\nu,p^\nu)$, the smooth solution to the Navier~Stokes system
 with the viscosity~$\nu$.
 
 \cole
 \begin{Proposition}[A priori estimates]
  \label{P.Ap}
  Any smooth solution $u$
  to \eqref{NSE0}--\eqref{hnavierbdry} defined on $[0,T]$,
  satisfies the following:
  \begin{itemize}
   \item[i.] If $\mu \in \mathbb{R}$, we have
   \begin{align}
    N^2(t)=
    (\Vert u(t)\Vert_{4}
    +\Vert u(t)\Vert_{2,\infty}
    +\Vert \nabla u(t)\Vert_{2}
    +\Vert \nabla u(t)\Vert_{L^\infty})^2   
    \le C \left(N^2(0) + \int_0^t N^3(s)\,ds\right)
    \label{ap1}
   \end{align}
   and
   \begin{align}
    \nu\int_0^t (\Vert \nabla u\Vert_{4}^2
    +\Vert D^2 u\Vert_{2}^2) \,ds
    \le C \left(N^2(0) + \int_0^t N^3(s)\,ds\right)
    ,\label{ap2}
   \end{align}
   for $t \in [0,T]$.
   \item[ii.] If $\mu \ge 0$, we have
   \begin{align}
    \bar{N}^2(t)=
    (\Vert u(t)\Vert_{4}
    +\Vert u(t)\Vert_{2,\infty}
    +\Vert \nabla u(t)\Vert_{1}
    +\Vert \nabla u(t)\Vert_{L^\infty})^2   
    \le C \left(\bar{N}^2(0) + \int_0^t \bar{N}^3(s)\,ds\right)
    \label{ap3}
   \end{align}
and
   \begin{align}
    \nu\int_0^t (\Vert \nabla u\Vert_{4}^2
    +\Vert D^2 u\Vert_{1}^2) \,ds
    \le C \left(\bar{N}^2(0) + \int_0^t \bar{N}^3(s)\,ds\right)
    ,\label{ap4}
   \end{align}
for $t \in [0,T]$.
\end{itemize}
\end{Proposition}
\colb
 
 We prove Proposition~\ref{P.Ap} in four steps. In Sections~\ref{sec.co}, \ref{sec.no},
 and \ref{sec.p}, we derive estimates for conormal derivatives of $u$, $\nabla u$,
 and~$\nabla p$.
 Then, in Section~\ref{sec.inf}, we perform $L^\infty$ estimates for $u$ and~$\nabla u$.
 We conclude the proof of Proposition~\ref{P.Ap}, Theorem~\ref{T01}, and
 Theorem~\ref{T03}(i) in Section~\ref{sec.apri}.
 Next, we establish 
 the inviscid limit in Section~\ref{sec.main} by proving Theorems~\ref{T02}
 and~\ref{T03}(ii). Finally, in Section~\ref{sec.e}, we sketch the proof of Theorem~\ref{T04}.
 
 \startnewsection{Conormal Derivative Estimates}{sec.co}
 
 In this section, we estimate~$\Vert u\Vert_{4}$.
 
 \cole
 \begin{Proposition}
  \label{P.Con}
  Let $u$ be a smooth solution of \eqref{NSE0}--\eqref{hnavierbdry} on~$[0,T]$.
  Then we have the inequality
  \begin{align}
   \Vert u(t)\Vert_{4}^2
   +
   c_0 \nu\int_0^t
   \Vert \nabla u(s)\Vert_{4}^2 \,ds
   \lec
   \Vert u_0\Vert_{4}^2
   +\int_0^t \bigl(
   \Vert u(s)\Vert_{4}^2 
   (\Vert u(s)\Vert_{W^{1,\infty}}+
   \Vert u(s)\Vert_{2,\infty}+1
   )+\Vert u(s)\Vert_{4}\Vert \nabla p(s)\Vert_{3} \bigr)\,ds
   ,
   \label{EQ.Con}
  \end{align}
where $c_0>0$ is independent of~$\nu$. 
 \end{Proposition}
 \colb
 
 \begin{proof}[Proof of Proposition~\ref{P.Con}]
  We proceed by induction on $0\le|\alpha|\le 4$ and
  first consider $|\alpha| = 0$,
which coincides with the usual $L^2$ estimate
  \begin{align}
   \frac{1}{2}\frac{d}{dt}\Vert u\Vert_{L^2}^2
   +\nu \Vert \nabla u\Vert_{L^2}^2
   = -2\mu \nu \Vert u_h\Vert_{L^2(\partial \Omega)}^2
   ,\label{L^2}
  \end{align}
where the boundary term results from 
$
   -\nu\int u\Delta u \,dx
   =\nu \Vert \nabla u\Vert_{L^2}^2
   +\int_{\partial \Omega} \partial_z u u \,d\sigma,
$
since $u \cdot n = 0$ and $n = (0,0,-1)$ on~$\partial \Omega$.
Using \eqref{hnavierbdry},
it follows that
  \begin{align}
   -\nu\int u\Delta u \,dx
   =\nu \Vert \nabla u\Vert_{L^2}^2
   +2\mu \nu \Vert u\Vert_{L^2(\partial \Omega)}^2,
   \label{b2L^2}
  \end{align}
  which is a coercive term, for~$\mu \ge 0$.
  On the other hand, when $\mu < 0$ we use the trace theorem 
  and Young's inequality to obtain
  \begin{align}
   2|\mu| \nu \Vert u\Vert_{L^2(\partial \Omega)}^2
   \le
   \epsilon \nu \Vert \nabla u\Vert_{L^2}^2
   +C_\epsilon \mu^2 \nu \Vert u\Vert_{L^2}^2
   \le
   \epsilon \nu \Vert \nabla u\Vert_{L^2}^2
   +C_\epsilon\Vert u\Vert_{L^2}^2
   ,\label{b3L^2}
  \end{align}
  where $\epsilon>0$ is arbitrarily small.
  Therefore, combining \eqref{L^2} and \eqref{b3L^2}
  and integrating in time, we get
  \begin{align}
   \Vert u(t)\Vert_{L^2}
   +c_0 \nu\int_0^t \Vert \nabla u\Vert_{L^2}^2 \,ds
   \lec
   \Vert u_0\Vert_{L^2} + \int_0^t \Vert u(s)\Vert_{L^2}^2 \,ds
   ,\label{0L^2}
  \end{align}
  concluding the base step.
  For the remaining part of the induction, we only 
  present the step from $|\alpha|=3$ to~$|\alpha|=4$.
  Thus, we assume that
  \begin{align}
   \Vert u\Vert_{3}^2 
   +c_0 \nu \int_0^t \Vert \nabla u\Vert_{3}^2 \,ds
   \lec \Vert u_0\Vert_{3}^2 + 
   \int_0^t \Bigl( \Vert u\Vert_{3}^2 
   (\Vert u\Vert_{W^{1,\infty}}+
   \Vert u\Vert_{2,\infty}+1
   )
   + \Vert u\Vert_{3}\Vert \nabla p\Vert_{2} \Bigr)\,ds
   \label{EQ02} 
  \end{align}
and aim to show~\eqref{EQ.Con}.
We first consider the case when all four conormal derivatives are horizontal, 
i.e., $Z^\alpha = Z^{\tilde{\alpha}}_\hh$ for $|\alpha|=4$, 
and then the general case. Applying $Z^\alpha$ to \eqref{NSE0}$_1$, we obtain
  \begin{align}
   \partial_t Z^{\tilde{\alpha}}_\hh u
   -\nu \Delta Z^{\tilde{\alpha}}_\hh u
   + u \cdot \nabla Z^{\tilde{\alpha}}_\hh u
   + \nabla Z^{\tilde{\alpha}}_\hh p
   =
   u \cdot \nabla Z^{\tilde{\alpha}}_\hh u
   -
   Z^{\tilde{\alpha}}_\hh(u \cdot \nabla u)
   ,
   \label{EQ01}
  \end{align}
  since $Z_h \nabla = \nabla Z_h$.
  We test \eqref{EQ01} with $Z^{\tilde{\alpha}}_\hh u$ and write
  \begin{align}
   \frac{1}{2}\frac{d}{dt}\Vert Z^{\tilde{\alpha}}_\hh u\Vert_{L^2}^2
   +\nu\Vert \nabla Z^{\tilde{\alpha}}_\hh u\Vert_{L^2}^2
   +2\mu \nu \Vert Z^{\tilde{\alpha}}_\hh u\Vert_{L^2(\partial \Omega)}^2
   =
   -
   ( u \cdot \nabla Z^{\tilde{\alpha}}_\hh u
   -Z^{\tilde{\alpha}}_\hh(u \cdot \nabla u)
   ,Z^{\tilde{\alpha}}_\hh u)
   .
   \label{EQ54}
  \end{align} 
  Note that to obtain \eqref{EQ54},
  we used incompressibility,
  while as in \eqref{b2L^2}, the diffusion term leads to 
  \begin{align}
   -\nu\int \Delta Z^{\tilde{\alpha}}_\hh u Z^{\tilde{\alpha}}_\hh u\,dx
   =\nu \Vert \nabla Z^{\tilde{\alpha}}_\hh u\Vert_{L^2}^2
   +\int_{\partial \Omega} \partial_z Z^{\tilde{\alpha}}_\hh u Z^{\tilde{\alpha}}_\hh u \,d\sigma
   =\nu \Vert \nabla Z^{\tilde{\alpha}}_\hh u\Vert_{L^2}^2
   +2\mu \nu \Vert Z^{\tilde{\alpha}}_\hh u\Vert_{L^2(\partial \Omega)}^2 
   ,\label{EQ142} 
  \end{align}
  where the last equality follows by 
  applying $Z^{\tilde{\alpha}}_\hh$ to~\eqref{hnavierbdry}.
  For the boundary term, as in \eqref{b3L^2}, we use the trace theorem 
  and Young's inequality to obtain, 
  \begin{align}
   |\mu| \nu \Vert Z^{\tilde{\alpha}}_\hh u\Vert_{L^2(\partial \Omega)}^2
   \le
   \epsilon \nu \Vert \nabla Z^{\tilde{\alpha}}_\hh u\Vert_{L^2}^2
   +C_\epsilon\Vert u\Vert_{4}^2
   ,\label{EQ144}
  \end{align}
  where $\epsilon>0$ is arbitrarily small.
  Therefore, we only have to estimate the commutator term,
  which we expand as
  \begin{align}
   u \cdot \nabla Z^{\tilde{\alpha}}_\hh u
   -
   Z^{\tilde{\alpha}}_\hh(u \cdot \nabla u)
   =
   -\sum_{1\le |\tilde{\beta}| \le 4} 
   {\tilde{\alpha} \choose \tilde{\beta}}
   (Z^{\tilde{\beta}}_\hh u_\hh 
   \cdot 
   \nabla_\hh Z^{\tilde{\alpha}-\tilde{\beta}}_\hh u
   + 
   Z^{\tilde{\beta}}_\hh u_3 
   \partial_z Z^{\tilde{\alpha}-\tilde{\beta}}_\hh u)
   .
   \label{EQ45}                              
  \end{align}
  For the terms involving $\nabla_\hh = (\partial_1,\partial_2)$, 
  we take the uniform norm of the factor with the lower number of derivatives, which gives
  \begin{align}
   \Vert 
   Z^{\tilde{\beta}}_\hh u_\hh 
   \cdot 
   \nabla_\hh Z^{\tilde{\alpha}-\tilde{\beta}}_\hh u
   \Vert_{L^2}
   \lec
   \Vert u\Vert_{2,\infty}\Vert u\Vert_{4}.
   \label{EQ30}
  \end{align} 
For the remaining terms in \eqref{EQ45},
  when $|\tilde{\beta}|\neq 4$, we first divide and multiply by $\varphi$ to write
  \begin{equation}
   \label{EQ143}
   \left\Vert Z_\hh^{\tilde{\beta}} \frac{u_3}{\varphi} Z_3 Z_\hh^{\tilde{\alpha} - {\tilde{\beta}}} u\right\Vert_{L^2}  
   \lec 
   \begin{cases}
    \bigl\Vert Z \frac{u_3}{\varphi}\bigr\Vert_{L^\infty}\Vert u\Vert_{4}
    \lec \Vert u\Vert_{2,\infty}\Vert u\Vert_{4}, & |{\tilde{\beta}}|=1 \\
    \bigl\Vert Z_\hh \frac{u_3}{\varphi}\bigr\Vert_{L^\infty}\Vert Z_\hh Z_3 u\Vert_{2}
    +\left\Vert Z_\hh \frac{u_3}{\varphi}\right\Vert_{2}\Vert Z_\hh Z_3 u\Vert_{L^\infty}
    \lec \Vert u\Vert_{4}\Vert u\Vert_{2,\infty}, & |{\tilde{\beta}}|=2 \\
    \left\Vert Z_\hh \frac{u_3}{\varphi}\right\Vert_{2}\Vert Z_\hh Z_3 u\Vert_{L^\infty}
    \lec \Vert u\Vert_{4}\Vert u\Vert_{2,\infty}, &   |{\tilde{\beta}}|=3,
   \end{cases}
  \end{equation}
  where we have used \eqref{EQ.int}
  upon taking $(f,g)=(Z_\hh \frac{u_3}{\varphi}, Z_\hh Z_3 u)$
  for $|\tilde{\beta}|=2$, and~\eqref{EQ.u3} to bound $Z_\hh \frac{u_3}{\varphi}$.
  Finally, when $|\tilde{\beta}| = 4$, we immediately have 
  \begin{align}
   \Vert Z_\hh^{\tilde{\beta}} u_3 \partial_z Z_\hh^{\tilde{\alpha} - {\tilde{\beta}}} u\Vert_{L^2}  
   \lec 
   \Vert u\Vert_{4}\Vert \partial_z u\Vert_{L^\infty}
   \comma   |{\tilde{\beta}}|=4
   .\label{EQ33} 
  \end{align}
  Now, combining \eqref{EQ54}--\eqref{EQ33} and
  absorbing the factor of $\nu\Vert \nabla Z^{\tilde{\alpha}}_\hh u\Vert_{L^2}^2$, we obtain
  \begin{align}
   \begin{split}
   \frac{d}{dt}\Vert Z^{\tilde{\alpha}}_\hh u\Vert_{L^2}^2 
   +
   c_0\nu
   \Vert \nabla Z^{\tilde{\alpha}}_\hh u\Vert_{L^2}^2
   &
   \lec
   \Vert u\Vert_{4}
   (\Vert u\Vert_{4}+ 
   \Vert 
   u \cdot \nabla Z^{\tilde{\alpha}}_\hh u
   -
   Z^{\tilde{\alpha}}_\hh(u \cdot \nabla u)
   \Vert_{L^2}
   )
   \\&
   \lec
   \Vert u\Vert_{4}^2
   (\Vert u\Vert_{2,\infty}+\Vert u\Vert_{W^{1,\infty}}+1)
   ,  
   \label{EQ38}   
  \end{split}
  \end{align}
  where $c_0 >0$ is independent of~$\nu$.
  Integrating \eqref{EQ38} implies 
  \begin{align}
   \Vert Z^{\tilde{\alpha}}_\hh u(t)\Vert_{L^2}^2 
   &+
   c_0\nu\int_0^t
   \Vert \nabla Z^{\tilde{\alpha}}_\hh u(s)\Vert_{L^2}^2\,ds
   \lec
   \Vert Z^{\tilde{\alpha}}_\hh u_0\Vert_{L^2}^2
   +\int_0^t \Vert u(s)\Vert_{4}^2
   (\Vert u(s)\Vert_{2,\infty}+\Vert u(s)\Vert_{W^{1,\infty}}+1)\,ds
   ,\llabel{EQ04}
  \end{align}
  for~$t \in [0,T]$.
  
  Proceeding to the general case,
  we consider $Z^\alpha = Z^{\tilde{\alpha}}_\hh Z^k_3$ 
  where $1\le k \le 4$, and apply it to \eqref{NSE0} obtaining
  \begin{align}
   Z^\alpha u_t
   - \nu \Delta Z^\alpha u
   + u\cdot \nabla Z^\alpha u
   = 
   u\cdot \nabla Z^\alpha u - Z^\alpha(u\cdot \nabla u) 
   - Z^\alpha \nabla p
   + \nu Z^\alpha \Delta u
   - \nu \Delta Z^\alpha u
   .
   \label{EQ46}
  \end{align}
  Upon testing with $Z^\alpha u$, 
  the left-hand side of \eqref{EQ46} gives
  \begin{align}
   \frac{1}{2}\frac{d}{dt}\Vert Z^\alpha u\Vert_{L^2}^2
   +\nu \Vert \nabla Z^\alpha u \Vert_{L^2}^2.
   \label{EQ55}
  \end{align}
  Observe that there is no boundary term arising from $-\nu \Delta Z^\alpha u$
  since $Z_3 = \varphi \partial_z = 0$ on~$\partial \Omega$.
  As for the right-hand side of \eqref{EQ46}, 
  we first expand the convection term as
  \begin{align}
   \begin{split}
    u \cdot \nabla Z^\alpha u - Z^\alpha(u \cdot \nabla u)
    &=
    u \cdot \nabla Z^\alpha u - u \cdot Z^\alpha \nabla u
    -
    \sum_{1 \le |\beta|\le |\alpha|}
    {\alpha \choose \beta}
    Z^\beta u \cdot Z^{\alpha-\beta} \nabla u
    \\&=
    u_3 \partial_z Z^\alpha u - u_3 Z^\alpha \partial_z u    
    -
    \sum_{1\le |\beta|\le |\alpha|}
    {\alpha \choose \beta}
    Z^\beta u \cdot Z^{\alpha-\beta} \nabla u
    \\&=
    -
    \sum_{j=0}^{k-1}
    \tilde{c}^k_{j,\varphi} u_3 \partial_z Z^{\tilde{\alpha}}_\hh Z^j_3 u  
    -
    \sum_{1\le |\beta|\le |\alpha|}
    {\alpha \choose \beta}
    Z^\beta u \cdot Z^{\alpha-\beta} \nabla u
    \\&= I_1 + I_2
    .
    \label{EQ13}        
   \end{split}
  \end{align}
  To estimate $I_1$, we divide and multiply by $\varphi$, obtaining
  \begin{align}
   \sum_{j=0}^{k-1}
   \left\Vert \tilde{c}^k_{j,\varphi} \frac{u_3}{\varphi} Z^{\tilde{\alpha}}_\hh Z^{j+1}_3 u\right\Vert_{L^2}  
   \lec
   \left\Vert \frac{u_3}{\varphi}\right\Vert_{L^\infty}\Vert u\Vert_{4}
   \lec
   \Vert \partial_z u\Vert_{L^\infty}\Vert u\Vert_{4}
   ,
   \label{EQ56}
  \end{align}
  using~\eqref{EQ.u3}.
  As for $I_2$, we rewrite  
  \begin{align}
   I_2 = -
   \sum_{1\le |\beta| \le |\alpha|}
   {\alpha \choose \beta} 
   (Z^{\beta} u_\hh 
   \cdot 
   \nabla_\hh Z^{\alpha-\beta} u
   + 
   Z^{\beta} u_3 
   Z^{\alpha-\beta} \partial_z u)
   = I_{21} + I_{22}
   .
   \label{EQ57}
  \end{align}
  The sum $I_{21}$ is estimated
  using the bounds in~\eqref{EQ30}.
  Next, to bound $I_{22}$, we proceed as in \eqref{EQ143}--\eqref{EQ33},
  which gives lower order terms when~$|\beta|\neq 4$. 
  In particular, since
  \begin{align}
   \frac{Z_3 u_3}{\varphi} = \nabla_\hh \cdot u_\hh,
   \llabel{EQ32}
  \end{align}
  and $1/\varphi$ commutes with $Z_h$, we only commute $\partial_z$ and $Z_3$ as in
  \eqref{EQL02}(i), obtaining lower order terms.
  Omitting further details, we conclude
  \begin{align}
   (u \cdot \nabla Z^\alpha u - Z^\alpha(u \cdot \nabla u),
   Z^\alpha u)
   \lec 
   \Vert u\Vert_{4}
   (\Vert u\Vert_{2,\infty}+\Vert u\Vert_{W^{1,\infty}}).
   \label{EQ58}
  \end{align}
  Returning to the right-hand side of \eqref{EQ46},
  it remains to estimate the pressure term and the commutator term including the Laplacian.
  We start by testing the pressure term against $Z^\alpha u$, obtaining
  \begin{align}
   (Z^\alpha \nabla p, Z^\alpha u)
   =
   (Z^\alpha \nabla_\hh p, Z^\alpha u_\hh)
   +
   (Z^\alpha \partial_z p, Z^\alpha u_3) 
   .
   \label{EQ47}
  \end{align}
  Since $Z_3$ and $\partial_z$ do not commute, we aim to introduce lower order terms and
  gain a derivative using incompressibility. To achieve this, we first use Lemma~\ref{L01}(i) and write
  \begin{align}
   (Z^k_3 \partial_z Z^{\tilde{\alpha}}_\hh p, Z^\alpha u_3)
   =
   (\partial_z Z^k_3  Z^{\tilde{\alpha}}_\hh p, Z^\alpha u_3)
   +
   \sum_{j=0}^{k-1}
   (c^k_{j,\varphi} \partial_z Z^j_3  Z^{\tilde{\alpha}}_\hh p, Z^\alpha u_3) 
   .
   \label{EQ48}
  \end{align}
  Note that each term under the sum is bounded by $\Vert \nabla p\Vert_{3}\Vert u\Vert_{4}$.
  For the leading order term, we integrate by parts and write
  \begin{align}
   (\partial_z Z^k_3 Z^{\tilde{\alpha}}_\hh p, Z^\alpha u_3) 
   =
   (Z^\alpha p, \partial_z Z^\alpha u_3)
   =
   (Z^\alpha p, Z^\alpha \partial_z  u_3)
   -
   \sum_{j=0}^{k-1}
   (c^k_{j,\varphi} Z^\alpha p, Z^j_3 Z^{\tilde{\alpha}}_\hh \partial_z u_3) 
   ,
   \label{EQ49}
  \end{align}
  as there is no boundary term due to $Z_3 f = 0$
  on $\partial \Omega$ for any sufficiently smooth function~$f$. 
  Moreover, each term under the sum is bounded by $\Vert \nabla p\Vert_{3}\Vert u\Vert_{4}$.
  Since $Z^\alpha \nabla_\hh \cdot u_\hh + Z^\alpha \partial_z u_3 = 0$,
  \eqref{EQ47}--\eqref{EQ49} imply
  \begin{align}
   (Z^\alpha \nabla p, Z^\alpha u)
   \lec 
   \Vert \nabla p\Vert_{3}\Vert u\Vert_{4}
   ,
   \label{EQ50}
  \end{align}
  leaving us only the commutator term for the Laplacian. 
  Using Lemma~\ref{L01}, we rewrite 
  \begin{align}
   \begin{split}
    \nu Z^\alpha \Delta u
    - \nu \Delta Z^\alpha u
    =&
    \nu Z^\alpha \partial_{zz} u - \nu \partial_{zz} Z^\alpha u
    =
    \nu Z^{\tilde{\alpha}}_\hh (Z_3^k \partial_{zz} u - \partial_{zz} Z_3^k u)
    \\
    =&
    \nu
    \sum_{j=0}^{k-1}
    \sum_{l=0}^{j}
    \left(
    c^j_{l,\varphi} c^k_{j,\varphi} \partial_{zz} Z^l_3 Z^{\tilde{\alpha}}_\hh u 
    +
    (c^j_{l,\varphi})' c^k_{j,\varphi} \partial_{z} Z^l_3 Z^{\tilde{\alpha}}_\hh u
    \right)
    \\
    &+
    \nu
    \sum_{l=0}^{k-1}
    \left(
    c^k_{l,\varphi}\partial_{zz} Z^l_3 Z^{\tilde{\alpha}}_\hh u  
    +
    (c^k_{l,\varphi})'\partial_{z} Z^l_3 Z^{\tilde{\alpha}}_\hh u 
    \right)
    \label{EQ51}
    .
   \end{split}
  \end{align}
  Note that the highest-order term in \eqref{EQ51}, i.e.,
  $(j,l)=(k-1,j)$ for the first sum and 
  $l=k-1$ for the second sum, is
  \begin{align}
   2\nu c^k_{k-1,\varphi}
   \partial_{zz} Z^{k-1}_3 Z^{\tilde{\alpha}}_\hh u,
   \llabel{EQ52}
  \end{align}
and when we integrate it against $Z^\alpha u$, we obtain
  \begin{align}
   \begin{split}
    2\nu
    \int c^k_{k-1,\varphi}
    \partial_{zz} Z^{k-1}_3 Z^{\tilde{\alpha}}_\hh u
    Z^\alpha u
    \,ds
    &=
    -2\nu
    \int 
    \partial_{z} Z^{k-1}_3 Z^{\tilde{\alpha}}_\hh u
    \left(
    c^k_{k-1,\varphi} \partial_z Z^\alpha u                     
    + 
    (c^k_{k-1,\varphi})' Z^\alpha u  
    \right)
    \,ds
    \\ &\lec
    \nu
    \Vert \partial_{z} Z^{k-1}_3 Z^{\tilde{\alpha}}_\hh u\Vert_{L^2}
    (\Vert \partial_z Z^\alpha u\Vert_{L^2}
    +
    \Vert u\Vert_{4})
    \\ &     \le
    C_\epsilon\nu\Vert \partial_{z} Z^{k-1}_3 Z^{\tilde{\alpha}}_\hh u\Vert_{L^2}^2
    +
    \epsilon\nu
    (\Vert \partial_z Z^\alpha u\Vert_{L^2}^2
    +
    \Vert u\Vert_{4}^2)
    ,
    \label{EQ53}
   \end{split}
  \end{align}
  where $\epsilon>0$ is an arbitrarily small parameter. 
  Repeating this procedure for the
  remaining terms in \eqref{EQ51} yields
  \begin{align}
   (\nu Z^\alpha \Delta u
   - \nu \Delta Z^\alpha u, Z^\alpha u)
   \le 
   \epsilon\nu\Vert \partial_z Z^\alpha u\Vert_{L^2}^2
   + 2C_\epsilon \nu \Vert \nabla u\Vert_{3}^2 + \Vert u\Vert_{4}^2
   .
   \label{EQ61}
  \end{align}
  Therefore, collecting \eqref{EQ46}, \eqref{EQ55}, \eqref{EQ58}, \eqref{EQ50}, \eqref{EQ61},
  absorbing all the factors of 
  $\nu\Vert \partial_z Z^\alpha u\Vert_{L^2}^2$, and integrating in time,
  we obtain 
  \begin{align}
   \begin{split}
    &\Vert Z^\alpha u(t)\Vert_{L^2}^2
    +c_0 \nu \int_0^t \Vert \nabla Z^\alpha u(s)\Vert_{L^2}^2 \,ds
    \\&
    \lec
    \Vert Z^\alpha u_0\Vert_{L^2}^2
    +\int_0^t (\nu\Vert \nabla u(s)\Vert_{3}^2  
    +\Vert u(s)\Vert_{4}^2 
    (\Vert u(s)\Vert_{W^{1,\infty}}+
    \Vert u(s)\Vert_{2,\infty}+1
    ) 
    + \Vert u(s)\Vert_{4}\Vert \nabla p(s)\Vert_{3}) \,ds
    .\label{EQ05}  
   \end{split}
  \end{align} 
 Note that the inductive hypothesis \eqref{EQ02} implies
  \begin{align}
   \nu \int_0^t \Vert \nabla u(s)\Vert_{3}^2 \,ds
   \lec
   \int_0^t \Bigl(\Vert u(s)\Vert_{3}^2 
   (\Vert u(s)\Vert_{W^{1,\infty}}+
   \Vert u(s)\Vert_{2,\infty}+1
   )  
   + \Vert u(s)\Vert_{3}\Vert \nabla p(s)\Vert_{2}\Bigr) \,ds
   .\llabel{EQ006}
  \end{align}
  Consequently, using this in \eqref{EQ05} and summing with \eqref{EQ02}
  yields~\eqref{EQ.Con}.
 \end{proof}
 
 \startnewsection{Normal Derivative Estimates}{sec.no}
 
 In this section, our goal is to estimate $\partial_z u$ in~$H^2_\cco(\Omega)$.
 However, as in~\cite{MR1}, instead of studying this term directly, it is more convenient
 to work with
 \begin{align}
  \eta = \omega_\hh - 2\mu u_\hh^\perp,
  \label{eta}
 \end{align}
 where $\omega = \curl u$ and 
 $u_\hh^\perp = (-u_2,u_1)^T$. 
 Observe that $\eta$ solves
 \begin{align}
  \begin{split}  
   \eta_t
   -\nu \Delta \eta
   +u\cdot \nabla \eta
   &=
   \omega\cdot \nabla u_\hh
   +2\mu \nabla_\hh^\perp p
   \comma x \in \Omega
   \\
   \eta &= 0
   \comma z=0
   ,
   \label{EQ.eta}   
  \end{split}
 \end{align}
 where $\nabla_\hh^\perp p = (-\partial_2,\partial_1)^T$.
 The main advantage of analyzing $\eta$, rather than $\omega$,
 is that it vanishes on the boundary.
Moreover, 
 \begin{align}
  \Vert \partial_z u\Vert_{m}
  \lec
  \Vert \eta\Vert_{m}+\Vert u\Vert_{m+1}
  \comma
  \Vert \partial_z u\Vert_{L^\infty}
  \lec
  \Vert \eta\Vert_{L^\infty}+\Vert u\Vert_{1,\infty},
  \label{EQ62}
 \end{align}
 showing that the first-order normal derivative of $u$
 is controlled by the conormal derivatives of $u$ and~$\eta$.
 Due to \eqref{EQ62}, we may replace the term 
 $\Vert u\Vert_{W^{1,\infty}}$
 in \eqref{EQ.Con} by~$\Vert \eta\Vert_{L^\infty}$.
 We are now ready to state the main result of this section.
 
 \cole
 \begin{Proposition}
  \label{P.Nor}
  Let $u$ be a smooth solution of \eqref{NSE0}--\eqref{hnavierbdry} on~$[0,T]$.
  Then we have the inequality
  \begin{align}
   \begin{split}
    &\Vert \eta(t)\Vert_{m}^2
    +
    c_0 \nu \int_0^t
    \Vert \nabla \eta(s)\Vert_{m}^2 \,ds
    \\&
    \lec
    \Vert \eta_0\Vert_{m}^2
    + \int_0^t (
    \Vert \eta\Vert_{m}^2
    (\Vert u\Vert_{2,\infty}+\Vert \eta\Vert_{L^\infty})
    +\Vert \eta\Vert_{m}
    (\Vert u\Vert_{m+2}\Vert \eta\Vert_{L^\infty}
    +\Vert u\Vert_{m+1}\Vert u\Vert_{1,\infty}
    +\Vert \nabla p\Vert_{m})
    \,ds
    \label{EQ.Nor}
   \end{split}
  \end{align}
on $[0,T]$ for $m=1,2$, where $c_0>0$ is independent of~$\nu$. 
\end{Proposition}
\colb
 
We note that \eqref{EQ.Nor} for $m=2$ is needed to prove the first part 
of Proposition~\ref{P.Ap}, while $m=1$ is needed for the second part.
 
 \begin{proof}[Proof of Proposition~\ref{P.Nor}]
  
  We proceed by induction on $0\le |\alpha| \le 2$.
  When $|\alpha| = 0$, the standard $L^2$ estimates yield
  \begin{align}
   \Vert \eta(t)\Vert_{L^2}^2
   +\nu \int_0^t
   \Vert \nabla \eta\Vert_{L^2}^2 \,ds
   \lec
   \Vert \eta_0\Vert_{L^2}^2
   +\int_0^t 
   (\Vert \omega(s)\Vert_{L^\infty}\Vert \nabla u(s)\Vert_{L^2}+\Vert \nabla p(s)\Vert_{L^2})
   \Vert \eta(s)\Vert_{L^2} \,ds
   .
   \label{EQ06}
  \end{align}
  Using \eqref{eta}, we obtain
  \begin{align}
   \Vert \omega_\hh\Vert_{L^\infty}\lec \Vert \eta\Vert_{L^\infty}+ \Vert u\Vert_{L^\infty}
   \text{ and }
   \Vert \omega_\hh\Vert_{m}\lec \Vert \eta\Vert_{m} + \Vert u\Vert_{m}
   \comma m \in \mathbb{N}_0\label{EQ07}
   ,
  \end{align}
  while for the normal component of the vorticity we have
  \begin{align}
   \Vert \omega_3\Vert_{L^\infty} \lec \Vert u\Vert_{1,\infty}
   \text{ and }
   \Vert \omega_3\Vert_{m} \lec \Vert u\Vert_{m+1}  
   \comma m \in \mathbb{N}_0,
   \label{EQ08}
  \end{align}
  since $\omega_3 = \partial_1 u_2 - \partial_2 u_1$. 
  Therefore, using \eqref{EQ62}, \eqref{EQ07}, and \eqref{EQ08} in \eqref{EQ06},
  we conclude
  \begin{align}
   \begin{split}
    &\Vert \eta(t)\Vert_{L^2}^2
    +\nu \int_0^t
    \Vert \nabla \eta\Vert_{L^2}^2 \,ds
    \\&
    \lec
    \Vert \eta_0\Vert_{L^2}^2
    +\int_0^t 
    (
    \Vert \eta\Vert_{L^2}^2
    (\Vert \eta\Vert_{L^\infty}\Vert u\Vert_{1,\infty}) 
    +\Vert \eta\Vert_{L^2} 
    (\Vert u\Vert_{1}\Vert \eta\Vert_{L^\infty}
    + \Vert u\Vert_{1}\Vert u\Vert_{1,\infty}
    +\Vert \nabla p\Vert_{L^2})
    \,ds
    .
    \llabel{EQ09}
   \end{split}
  \end{align}
  
  It remains to establish similar inequalities for $Z^\alpha \eta$,
  where~$|\alpha|=1,2$. We only present our estimates for $|\alpha|=2$, assuming
  \eqref{EQ.Nor} for~$m=1$.We first
  consider the case with only tangential derivatives.\\
  
  \noindent\texttt{Step~1.} The case $Z^\alpha = Z^{\tilde{\alpha}}_\hh$. \\
  
  Applying $Z^{\tilde{\alpha}}_\hh$ to \eqref{EQ.eta}$_1$, we obtain
  \begin{align}
   Z^{\tilde{\alpha}}_\hh \eta_t
   -\nu \Delta Z^{\tilde{\alpha}}_\hh \eta
   +u\cdot \nabla Z^{\tilde{\alpha}}_\hh \eta
   &=
   Z^{\tilde{\alpha}}_\hh (\omega\cdot \nabla u_\hh)
   +2\mu Z^{\tilde{\alpha}}_\hh \nabla_\hh^\perp p
   +u\cdot \nabla Z^{\tilde{\alpha}}_\hh \eta
   -Z^{\tilde{\alpha}}_\hh (u \cdot \nabla \eta)
   .
   \label{EQ63}
  \end{align}
  Upon testing with $Z^{\tilde{\alpha}}_\hh \eta$,
  the left-hand side of \eqref{EQ63} gives
  \begin{align}
   \frac{1}{2}\Vert Z^{\tilde{\alpha}}_\hh \eta\Vert_{L^2}^2
   +\nu\Vert \nabla Z^{\tilde{\alpha}}_\hh \eta\Vert_{L^2}^2
   \label{EQ64}
  \end{align}
since there is no boundary term thanks to~\eqref{EQ.eta}$_2$.
  For the right-hand side of \eqref{EQ63}, 
  the pressure term yields
  \begin{align}
   (2\mu Z^{\tilde{\alpha}}_\hh \nabla_\hh^\perp p, Z^{\tilde{\alpha}}_\hh \eta)
   \lec
   \Vert \nabla p\Vert_{2}\Vert \eta\Vert_{2}
   ,
   \label{EQ65}
  \end{align}
  since~$|\alpha|=2$.
  For the term involving vorticity, we use
  \eqref{EQ.int} and write
  \begin{align}
   \begin{split}
    (Z^{\tilde{\alpha}}_\hh (\omega\cdot \nabla u_\hh),Z^{\tilde{\alpha}}_\hh \eta)
    \lec&
    \Vert Z^{\tilde{\alpha}}_\hh (\omega\cdot \nabla u_\hh)\Vert_{L^2}\Vert \eta\Vert_{2}
    \lec
    (\Vert \omega\Vert_{L^\infty}\Vert  \nabla u_\hh\Vert_{2}
    +\Vert \omega\Vert_{2}\Vert \nabla u_\hh\Vert_{L^\infty})\Vert \eta\Vert_{2}
    \\
    \lec&
    \Vert \eta\Vert_{2}^2
    (\Vert \eta\Vert_{L^\infty}+\Vert u\Vert_{1,\infty})
    +\Vert \eta\Vert_{2}\Vert u\Vert_{3}
    (\Vert \eta\Vert_{L^\infty}+\Vert u\Vert_{1,\infty}) 
    ,
    \label{EQ66}
   \end{split}
  \end{align}
  where, in the last inequality we have used \eqref{EQ07} and~\eqref{EQ08}.
  Finally, for the commutator term in \eqref{EQ63},
  we as in~\eqref{EQ45} first separate the normal and tangential derivatives. 
  Then, we divide and multiply 
  by $\varphi$ the terms involving $\partial_z$ obtaining
  \begin{align}
   I_\text{com} = u\cdot \nabla Z^{\tilde{\alpha}}_\hh \eta
   -Z^{\tilde{\alpha}}_\hh (u \cdot \nabla \eta)
   =
   \sum_{1\le |\tilde{\beta}| \le 2}
   {\tilde{\alpha} \choose \tilde{\beta}} 
   \left(Z^{\tilde{\beta}}_\hh u_\hh 
   \cdot 
   \nabla_\hh Z^{\tilde{\alpha}-\tilde{\beta}}_\hh \eta
   + 
   Z^{\tilde{\beta}}_\hh \frac{u_3}{\varphi}  
   Z_3 Z^{\tilde{\alpha}-\tilde{\beta}}_\hh \eta\right)
   .
   \label{EQ68}
  \end{align}
  Once again, we use \eqref{EQ.int} and estimate the sum in \eqref{EQ68}
  \begin{align}
   \begin{split}
    (I_\text{com}, Z^{\tilde{\alpha}}_\hh \eta)
    &\lec
    \left(\Vert Z_\hh u\Vert_{L^\infty}+\left\Vert Z_\hh \frac{u_3}{\varphi}\right\Vert_{L^\infty}\right)
    \Vert \eta\Vert_{2}^2
    +
    \left(\Vert Z_\hh u\Vert_{2}+\left\Vert Z_\hh \frac{u_3}{\varphi}\right\Vert_{2}\right)
    \Vert \eta\Vert_{L^\infty}\Vert \eta\Vert_{2}
    \\&\lec
    \Vert \eta\Vert_{2}^2\Vert u\Vert_{2,\infty}
    +
    \Vert \eta\Vert_{2}\Vert u\Vert_{4}\Vert \eta\Vert_{L^\infty}
    ,
    \label{EQ69}     
   \end{split}
  \end{align}
  where we have also used \eqref{EQ.u3}.
  Now, collecting \eqref{EQ64}--\eqref{EQ69}
  and integrating in time,
  it follows that 
  \begin{align}
   \begin{split}
    &\Vert Z^{\tilde{\alpha}}_\hh \eta(t)\Vert_{L^2}^2
    +\nu \int_0^t \Vert \nabla Z^{\tilde{\alpha}}_\hh \eta\Vert_{L^2}^2 \,ds
    \\&
    \lec 
    \int_0^t (
    \Vert \eta\Vert_{2}^2(\Vert \eta\Vert_{L^\infty}+\Vert u\Vert_{2,\infty})
    +\Vert \eta\Vert_{2}(\Vert u\Vert_{4}\Vert \eta\Vert_{L^\infty}
    +\Vert u\Vert_{3}\Vert u\Vert_{1,\infty}+\Vert \nabla p\Vert_{2})
    ) \,ds,
    \label{EQ72}
   \end{split}
  \end{align}
  concluding Step~1.\\
  
  \noindent\texttt{Step~2.} We now consider $Z^\alpha = Z^{\tilde{\alpha}}_\hh Z^k_3$,  
  for $1 \le k \le |\alpha| \le 2 $.\\
  
  Applying $Z^\alpha$ to the \eqref{EQ.eta}$_1$, we obtain
  \begin{align}
   \begin{split}
    Z^\alpha \eta_t
    -\nu \Delta Z^\alpha \eta
    +u\cdot \nabla Z^\alpha \eta
    =&
    Z^\alpha (\omega\cdot \nabla u_\hh)
    +2\mu Z^\alpha \nabla_\hh^\perp p
    +u\cdot \nabla Z^\alpha \eta
    -Z^\alpha (u \cdot \nabla \eta)
    \\&
    -\nu Z^\alpha \Delta \eta
    +\nu \Delta Z^\alpha \eta
    .
    \label{EQ73}
   \end{split}
  \end{align}
  Upon testing with $Z^\alpha \eta$, the left-hand side of the above equation gives
  \begin{align}
   \frac{1}{2}\frac{d}{dt}\Vert Z^\alpha \eta\Vert_{L^2}^2
   +\nu\Vert \nabla Z^\alpha \eta\Vert_{L^2}^2
   ,\label{EQ74}
  \end{align}
  since there is no boundary term due to the presence of~$Z_3$. Now, for the
  vorticity and the pressure terms, we repeat \eqref{EQ65} and \eqref{EQ66}, obtaining
  \begin{align}
   \begin{split}
    (Z^\alpha (\omega\cdot \nabla u_\hh)
    +2\mu Z^\alpha \nabla_\hh^\perp p,
    Z^\alpha \eta)
    \lec&
    \Vert \eta\Vert_{2}^2
    (\Vert \eta\Vert_{L^\infty}+\Vert u\Vert_{1,\infty})
    \\
    &+\Vert \eta\Vert_{2}
    (\Vert u\Vert_{3}\Vert \eta\Vert_{L^\infty}
    +\Vert u\Vert_{3}\Vert u\Vert_{1,\infty}
    +\Vert \nabla p\Vert_{2}) 
    .
    \label{EQ75}
   \end{split}
  \end{align}
  The first commutator term in \eqref{EQ73} is expanded as in 
  \eqref{EQ13} and \eqref{EQ57} by writing
  \begin{align}
   \begin{split}
    u \cdot \nabla Z^\alpha \eta - Z^\alpha(u \cdot \nabla \eta)
    =&
    u_3 \partial_z Z^\alpha \eta - u_3 Z^\alpha \partial_z \eta    
    -
    \sum_{1\le |\beta|\le |\alpha|}
    {\alpha \choose \beta}
    Z^\beta u \cdot Z^{\alpha-\beta} \nabla \eta
    \\=&
    -
    \sum_{j=0}^{k-1}
    \tilde{c}^k_{j,\varphi} u_3 \partial_z Z^{\tilde{\alpha}}_\hh Z^j_3 \eta  
    -
    \sum_{1\le |\beta|\le |\alpha|}
    {\alpha \choose \beta}
    Z^\beta u \cdot Z^{\alpha-\beta} \nabla \eta
    \\=&
    -
    \sum_{j=0}^{k-1}
    \tilde{c}^k_{j,\varphi} u_3 \partial_z Z^{\tilde{\alpha}}_\hh Z^j_3 \eta  
    -
    \sum_{1\le |\beta| \le |\alpha|}
    {\alpha \choose \beta} 
    Z^{\beta} u_\hh 
    \cdot 
    \nabla_\hh Z^{\alpha-\beta} \eta
    \\&-
    \sum_{1\le |\beta| \le |\alpha|}
    {\alpha \choose \beta}  
    Z^{\beta} u_3 
    Z^{\alpha-\beta} \partial_z \eta
    \\=& I_1 + I_2 + I_3
    .
    \label{EQ76}        
   \end{split}
  \end{align}
  We divide and multiply $I_1$ by $\varphi$ and estimate it as
  \begin{align}
   (I_1, Z^\alpha \eta) \lec
   \left\Vert \frac{u_3}{\varphi}\right\Vert_{L^\infty}
   \Vert \eta\Vert_{2}^2
   \lec \Vert \eta\Vert_{2}^2\Vert u\Vert_{1,\infty}.
   \label{EQ77}
  \end{align} 
  Next, we bound $I_2$ as
  \begin{align}
   (I_2, Z^\alpha \eta) \lec \Vert \eta\Vert_{2}^2\Vert u\Vert_{2,\infty}
   .\label{EQ78}
  \end{align}
  Lastly, for $I_3$, we first divide and multiply by $\varphi$
  and introduce lower order terms. It follows that
  \begin{align}
   Z^{\beta} u_3 
   Z^{\alpha-\beta} \partial_z \eta
   =
   Z^{\beta} \frac{u_3}{\varphi} 
   Z^{\alpha-\beta} Z_3 \eta
   + I_l,
   \label{EQ79}
  \end{align}
  where $I_l$ stands for the lower order terms.
  We now proceed as in \eqref{EQ69},
  obtaining through \eqref{EQ.int}
  \begin{align}
   (I_3, Z^\alpha \eta)\lec 
   \Vert \eta\Vert_{2}^2\Vert u\Vert_{2,\infty}
   +
   \Vert \eta\Vert_{2}\Vert u\Vert_{4}\Vert \eta\Vert_{L^\infty}
   .\label{EQ80}
  \end{align}
  Collecting \eqref{EQ76}--\eqref{EQ80}, we have the bound
  \begin{align}
   (u \cdot \nabla Z^\alpha \eta - Z^\alpha(u \cdot \nabla \eta)
   , Z^\alpha \eta) \lec 
   \Vert \eta\Vert_{2}^2\Vert u\Vert_{2,\infty}
   +\Vert \eta\Vert_{2}\Vert u\Vert_{4}\Vert \eta\Vert_{L^\infty}.
   \label{EQ81}
  \end{align}
  Now, it only remains to estimate the commutator term involving the Laplacian. 
  We use Lemma~\ref{L01} to write
  \begin{align}
   \begin{split}
    \nu Z^\alpha \Delta \eta
    - \nu \Delta Z^\alpha \eta
    =&
    \nu
    \sum_{j=0}^{k-1}
    \sum_{l=0}^{j}
    \left(
    c^j_{l,\varphi} c^k_{j,\varphi} \partial_{zz} Z^l_3 Z^{\tilde{\alpha}}_\hh \eta 
    +
    (c^j_{l,\varphi})' c^k_{j,\varphi} \partial_{z} Z^l_3 Z^{\tilde{\alpha}}_\hh \eta
    \right)
    \\
    &+
    \nu
    \sum_{l=0}^{k-1}
    \left(
    c^k_{l,\varphi}\partial_{zz} Z^l_3 Z^{\tilde{\alpha}}_\hh \eta  
    +
    (c^k_{l,\varphi})'\partial_{z} Z^l_3 Z^{\tilde{\alpha}}_\hh  \eta
    \right)
    \label{EQ82}
    ,
   \end{split}
  \end{align}
  with the highest-order term
  \begin{align}
   2\nu c^k_{k-1,\varphi}
   \partial_{zz} Z^{k-1}_3 Z^{\tilde{\alpha}}_\hh \eta
   .
   \llabel{EQ83}
  \end{align}
  When we integrate this term against $Z^\alpha \eta$
  and repeat the steps in \eqref{EQ53}, we obtain
  \begin{align}
   (2\nu c^k_{k-1,\varphi}
   \partial_{zz} Z^{k-1}_3 Z^{\tilde{\alpha}}_\hh \eta
   ,Z^\alpha \eta)
   \le
   C_\epsilon\nu\Vert \partial_{z} Z^{k-1}_3 Z^{\tilde{\alpha}}_\hh \eta\Vert_{L^2}^2
   +
   \epsilon\nu
   (\Vert \partial_z Z^\alpha \eta\Vert_{L^2}^2
   +
   \Vert \eta\Vert_{2}^2)
   ,
   \label{EQ84}
  \end{align}
  with $\epsilon >0$ arbitrarily small.
  Dealing with the remaining terms in \eqref{EQ82}, we get
  \begin{align}
   (\nu Z^\alpha \Delta \eta
   - \nu \Delta Z^\alpha \eta, Z^\alpha \eta)
   \le
   \epsilon \nu \Vert \partial_z Z^\alpha \eta\Vert_{L^2}^2
   +C_\epsilon (\Vert \partial_z \eta\Vert_{1}^2 + \Vert \eta\Vert_{2}^2)
   .\label{EQ11}
  \end{align}
  Now combining \eqref{EQ74}, \eqref{EQ75}, \eqref{EQ81},
  \eqref{EQ11} and integrating in time, we obtain
  \begin{align}
   \begin{split}
    &
    \Vert Z^\alpha \eta(t)\Vert_{L^2}^2
    +c_0\nu \int_0^t \Vert \nabla \eta\Vert_{L^2}^2\,ds
    \lec \Vert Z^\alpha \eta_0\Vert_{L^2}^2
    +\nu \int_0^t \Vert \partial_z \eta\Vert_{1}^2 \,ds
    \\&+
    \int_0^t (
    \Vert \eta\Vert_{2}^2
    (\Vert \eta\Vert_{L^\infty}+\Vert u\Vert_{2,\infty})
    +\Vert \eta\Vert_{2}
    (\Vert u\Vert_{4}\Vert \eta\Vert_{L^\infty}
    +\Vert u\Vert_{3}\Vert u\Vert_{1,\infty}
    +\Vert \nabla p\Vert_{2})\,ds
    .\label{EQ12}
   \end{split}
  \end{align}
  Consequently, using the inductive hypothesis, i.e., \eqref{EQ.Nor} for $m=1$
  to control the normal derivative of $\eta$,
  we obtain~\eqref{EQ.Nor}.
 \end{proof}
 
 \startnewsection{Pressure Estimates}{sec.p}
 
 In this section, we estimate the pressure term.
 We note that $\Vert \nabla p\Vert_{3}$
 is needed for the conormal and normal estimates, while 
 $\Vert D^2 p\Vert_{3}$ is required to bound
 $\Vert u\Vert_{2,\infty}$.
 
 \cole
 \begin{Proposition}
  \label{P.Pre}
  Let $(u,p)$ be a smooth solution of \eqref{NSE0}--\eqref{hnavierbdry}.
  Then we have the inequality
  \begin{align}
   \Vert D^2 p\Vert_{3}+\Vert \nabla p\Vert_{3}
   \lec 
   \Vert u\Vert_{4}(\Vert u\Vert_{2,\infty}+\Vert \eta\Vert_{L^\infty}+1)+\nu \Vert \nabla u\Vert_{4}
   ,
   \label{EQ.Pre}
  \end{align}
where the implicit constant is independent of~$\nu$. 
 \end{Proposition}
 \colb
 
 Observe that upon applying the divergence to \eqref{NSE0}, we  obtain
 \begin{align}
  -\Delta p = \partial_i u_j \partial_j u_i
  ,\label{pre} 
 \end{align}
 while on the boundary, using \eqref{hnavierbdry}, we have
 \begin{align}
  \nabla p \cdot n = -\partial_z p
  = -\nu \partial_{ii} u \cdot n
  = \nu \partial_{ii} u_3 
  = \nu \partial_{zz} u_3
  = -\nu \partial_z \nabla_\hh \cdot u_\hh
  = -2\mu \nu \nabla_\hh \cdot u_\hh.
  \label{pre.bdry}
 \end{align}
 
 We start by treating $ Z^{\tilde{\alpha}}_\hh p$ and then
 we use the induction on $k$ to cover the general case,
 recalling that $Z^\alpha = Z^{\tilde{\alpha}}_\hh Z^k_3$.
 
 \cole
 \begin{Lemma}
  \label{L.Pre}
  Let $(u,p)$ be a smooth solution of \eqref{NSE0}--\eqref{hnavierbdry}.
  Then we have the inequality
  \begin{align}
   \Vert D^2 Z^{\tilde{\alpha}}_\hh p\Vert_{L^2}+
   \Vert \nabla Z^{\tilde{\alpha}}_\hh p\Vert_{L^2}
   \lec 
   \Vert u\Vert_{4}(\Vert u\Vert_{2,\infty}+\Vert \eta\Vert_{L^\infty}+1)
   + \nu\Vert \nabla u\Vert_{|\tilde{\alpha}|+1},
   \label{EQ.hPre}
  \end{align}
  where $0 \le |\tilde{\alpha}|\le 3$.
 \end{Lemma}
 \colb
 
 \begin{proof}[Proof of Lemma~\ref{L.Pre}]
  
  We first present our estimates for $|\tilde{\alpha}|=3$.
  We start by applying $Z^{\tilde{\alpha}}_\hh $ to \eqref{pre} and~\eqref{pre.bdry}. 
  It follows that $Z^{\tilde{\alpha}}_\hh p$ solves
  \begin{align}
   -\Delta Z^{\tilde{\alpha}}_\hh p = Z^{\tilde{\alpha}}_\hh (\partial_i u_j \partial_j u_i)
   ,
   \llabel{hpre}
  \end{align}
  with the boundary condition
  \begin{align}
   \nabla Z^{\tilde{\alpha}}_\hh p \cdot n
   = -2\mu \nu Z^{\tilde{\alpha}}_\hh \nabla_\hh \cdot u_\hh.
   \llabel{hpre.bdry}
  \end{align}
  Applying the $H^2$ elliptic estimate for the Neumann problem (see~\cite{G}), we obtain
  \begin{align}
   \Vert D^2 Z^{\tilde{\alpha}}_\hh p\Vert_{L^2}
   +
   \Vert \nabla Z^{\tilde{\alpha}}_\hh p\Vert_{L^2}
   \lec
   \Vert Z^{\tilde{\alpha}}_\hh (\partial_i u_j \partial_j u_i)\Vert_{L^2}
   +
   2|\mu| \nu \Vert Z^{\tilde{\alpha}}_\hh \nabla_\hh \cdot u_\hh\Vert_{H^1}
   .\label{EQ87}
  \end{align}
  Note that
  to estimate $Z^{\tilde{\alpha}}_\hh (\partial_i u_j \partial_j u_i)$
  it suffices to consider 
  $Z^{\tilde{\alpha}}_\hh (Z_\hh u_\hh Z_\hh u_\hh )$ 
  and 
  $Z^{\tilde{\alpha}}_\hh (Z_\hh u_3 \partial_z u_\hh)$  
  since we may use the divergence-free condition to eliminate 
  the normal derivatives in $\partial_z u_3 \partial_z u_3$.
  For the first case, we use \eqref{EQ.int} and obtain
  \begin{align}
   \Vert Z^{\tilde{\alpha}}_\hh (Z_\hh u_\hh Z_\hh u_\hh )\Vert_{L^2}
   \lec
   \Vert u\Vert_{4}\Vert u\Vert_{1,\infty},
   \llabel{EQ89}
  \end{align}
  while for the second case, we first expand by writing
  \begin{align}
   Z^{\tilde{\alpha}}_\hh (Z_\hh u_3 \partial_z u_\hh)=
   \sum_{0\le |\tilde{\beta}|\le |\tilde{\alpha}|}
   {\alpha \choose \beta}
   Z^{\tilde{\beta}}_\hh Z_\hh u_3
   Z^{\tilde{\alpha}-\tilde{\beta}}_\hh \partial_z u_\hh
   .\label{EQ90}
  \end{align}
  When $|\tilde{\beta}|=3$, we obtain
  \begin{align}
   \Vert Z_\hh^{\tilde{\beta}} Z_\hh u_3 \partial_z Z_\hh^{\tilde{\alpha} - {\tilde{\beta}}} u_\hh\Vert_{L^2}\lec
   \Vert u\Vert_{4}(\Vert \eta\Vert_{L^\infty}+\Vert u\Vert_{1,\infty})
   .\label{EQ91a}
  \end{align} 
  For the values of $|\tilde{\beta}|\neq 3$,
  we divide and multiply by $\varphi$ and write
  \begin{equation}
   \label{EQ91b}
   \left\Vert Z_\hh^{\tilde{\beta}} Z_\hh \frac{u_3}{\varphi} 
   Z_3 Z_\hh^{\tilde{\alpha} - {\tilde{\beta}}} u_\hh\right\Vert_{L^2}
   \lec
   \begin{cases}
    \bigl\Vert Z_\hh \frac{u_3}{\varphi}\bigr\Vert_{L^\infty}\Vert Z^{\tilde{\alpha}-\tilde{\beta}}_\hh Z_3 u\Vert_{L^2}
    \lec \Vert u\Vert_{4}\Vert u\Vert_{2,\infty}, & |{\tilde{\beta}}|=0
    \\
    \left(\bigl\Vert Z_\hh \frac{u_3}{\varphi}\bigr\Vert_{L^\infty}\Vert Z_\hh Z_3 u\Vert_{2}
    +\left\Vert Z_\hh \frac{u_3}{\varphi}\right\Vert_{2} \Vert Z_\hh Z_3 u\Vert_{L^\infty}\right)
    \lec\Vert u\Vert_{4}\Vert u\Vert_{2,\infty}, & |{\tilde{\beta}}|=1
    \\
    \Vert u\Vert_{4}\Vert u\Vert_{2,\infty}, & |{\tilde{\beta}}|=2,
   \end{cases}
  \end{equation}
  where, for $|\tilde{\beta}|= 1$, we have used \eqref{EQ.int}
  by taking $f$ and $g$ as $Z_\hh \frac{u_3}{\varphi}$ and $Z_\hh Z_3 u$, respectively. 
  
  Now, it only remains to estimate the last term in \eqref{EQ87},
  where we have
  \begin{align}
   \nu \Vert Z^{\tilde{\alpha}}_\hh \nabla_\hh \cdot u_\hh\Vert_{H^1}
   \lec
   \nu\Vert u\Vert_{4}+ \nu\Vert \nabla u\Vert_{|\tilde{\alpha}|+1}
   .\label{EQ92}
  \end{align}
  As a consequence, combining \eqref{EQ87}--\eqref{EQ92}
  gives \eqref{EQ.hPre}, for~$|\tilde{\alpha}|=3$.
  When we consider a lower order conormal derivative, i.e., 
  $|\tilde{\alpha}|\le 2$, the only difference in our estimates is
  the treatment of the boundary term. Therefore, upon performing
  \eqref{EQ87}--\eqref{EQ92}, for $|\tilde{\alpha}|\le 2$, we may conclude~\eqref{EQ.hPre}.
 \end{proof}
 
 Now, we move on to the general case.
 
 \begin{proof}[Proof of Proposition~\ref{P.Pre}]
  
  We present our estimates for
  $\Vert D^2 Z^\alpha p\Vert_{L^2}$ and
   $\Vert \nabla Z^\alpha p\Vert_{L^2}$ with $|\alpha|=3$,
  using induction on $0\le k \le 3$ satisfying 
  $k + |\tilde{\alpha}| = 3$. The case $k=0$ follows by using Lemma~\ref{L.Pre}.
  Thus, we assume 
  \begin{align}
   \begin{split}
    \Vert D^2 Z^{\tilde{\beta}}_\hh Z^{k}_3 p\Vert_{L^2}+
    \Vert \nabla Z^{\tilde{\beta}}_\hh Z^{k}_3 p\Vert_{L^2}
    \lec& \Vert u\Vert_{4}(\Vert u\Vert_{2,\infty}+\Vert \eta\Vert_{L^\infty}+1) 
    \\&+
    \nu \Vert \nabla \nabla_\hh Z^{\tilde{\beta}}_\hh u\Vert_{L^2}
    \comma |\tilde{\beta}| + k = 3
    \label{EQ.pre.k-1} 
   \end{split}
  \end{align}
  and aim to prove that
  \begin{align}
   \begin{split}
    \Vert D^2 Z^{\tilde{\alpha}}_\hh Z^{k+1}_3 p\Vert_{L^2}+
    \Vert \nabla Z^{\tilde{\alpha}}_\hh Z^{k+1}_3 p\Vert_{L^2}
    \lec& \Vert u\Vert_{4}(\Vert u\Vert_{2,\infty}+\Vert \eta\Vert_{L^\infty}+1) 
    \\
    &+
    \nu \Vert \nabla \nabla_\hh Z^{\tilde{\alpha}}_\hh u\Vert_{L^2}
    \comma |\tilde{\alpha}| + k + 1 = 3.
    \label{EQ.pre.k}
   \end{split}
  \end{align}
  Now, applying $Z^\alpha = Z^{\tilde{\alpha}}_\hh Z^{k+1}_3$ to \eqref{pre}, it follows that
  \begin{align}
   -\Delta Z^\alpha p  
   = Z^\alpha(\partial_i u_j \partial_j u_i)
   + Z^\alpha \Delta p -\Delta Z^\alpha p
   .\llabel{kpre}  
  \end{align}
  To find the boundary condition, we write
  \begin{align}
   \nabla Z^\alpha p \cdot n
   = -\partial_z Z^{\tilde{\alpha}}_\hh Z^{k+1}_3 p
   = -Z^{\tilde{\alpha}}_\hh \sum_{j=0}^{k}
   \tilde{c}_{j,\varphi}^{k+1} Z_3^j \partial_z p 
   = -\tilde{c}_{0,\varphi}^{k+1} Z^{\tilde{\alpha}}_\hh \partial_z p 
   = -2\mu \nu \tilde{c}_{0,\varphi}^{k+1} Z^{\tilde{\alpha}}_\hh \nabla_\hh \cdot u_\hh    
   ,\llabel{kpre.bdry}
  \end{align}
  where we have used \eqref{pre.bdry} and the fact that $Z_3$ vanishes on the boundary.
  Applying the elliptic estimates, as in \eqref{EQ92}, we obtain
  \begin{align}
   \Vert D^2 Z^\alpha p\Vert_{L^2}+
   \Vert \nabla Z^\alpha p\Vert_{L^2}
   \lec
   \Vert Z^\alpha(\partial_i u_j \partial_j u_i)\Vert_{L^2}
   +
   \Vert Z^\alpha \Delta p -\Delta Z^\alpha p\Vert_{L^2}
   +
   |\mu| \nu \Vert Z^{\tilde{\alpha}}_\hh \nabla_\hh \cdot u_\hh\Vert_{H^1}
   .\label{EQ93}
  \end{align}
  For the first term on the right-hand side of the above inequality,
  we introduce lower order terms when necessary and use the estimates
  in \eqref{EQ91a} and \eqref{EQ91b} to conclude that
  \begin{align}
   \Vert Z^\alpha(\partial_i u_j \partial_j u_i)\Vert_{L^2}
   +|\mu| \nu \Vert Z^{\tilde{\alpha}}_\hh \nabla_\hh \cdot u_\hh\Vert_{H^1}
   \lec 
   \Vert u\Vert_{4}(\Vert u\Vert_{2,\infty}+\Vert \eta\Vert_{L^\infty}+1)
   +\nu \Vert \nabla \nabla_\hh Z^{\tilde{\alpha}}_\hh u\Vert_{L^2}
   .
   \label{EQ94}
  \end{align}
  Therefore, it only remains to estimate the commutator term for the pressure.
  Using Lemma~\ref{L01}, we may expand this term as
  \begin{align}
   \begin{split}
    Z^\alpha \Delta p -\Delta Z^\alpha p
    &=Z^\alpha \partial_{zz} p - \partial_{zz}Z^\alpha p
    = Z^\alpha \partial_{zz} p -
    \sum_{j=0}^{k+1} \sum_{l=0}^{j}
    \tilde{c}^j_{l,\varphi} \tilde{c}^{k+1}_{j,\varphi} Z^l_3 Z^{\tilde{\alpha}}_\hh \partial_{zz} p 
    +
    \sum_{j=0}^{k+1}
    (\tilde{c}^{k+1}_{j,\varphi})' Z^j_3 Z^{\tilde{\alpha}}_\hh \partial_z p 
    \\&=
    -\sum_{j=0}^{k} \sum_{l=0}^{j}
    \tilde{c}^j_{l,\varphi} \tilde{c}^{k+1}_{j,\varphi} Z^l_3 Z^{\tilde{\alpha}}_\hh \partial_{zz} p 
    -
    \sum_{l=0}^{k}
    \tilde{c}^{k+1}_{l,\varphi} Z^l_3 Z^{\tilde{\alpha}}_\hh \partial_{zz} p 
    +
    \sum_{j=0}^{k}
    (\tilde{c}^{k+1}_{j,\varphi})' Z^j_3 Z^{\tilde{\alpha}}_\hh \partial_z p
    \\&= I_1 + I_2 + I_3     
    ,\llabel{EQ95}
   \end{split}
  \end{align}
  where we have used that $(\tilde{c}_{k+1,\varphi}^{k+1})' = (1)' = 0$.
  We shall only treat the highest-order terms in the above equation. Observe that for $I_3$
  this term is $\partial_z Z^{\tilde{\alpha}}_\hh Z^{k}_3 p$ which is
  of lower order and is thus bounded as in~\eqref{EQ.pre.k-1}.
  The highest-order terms of $I_1$ and $I_2$ are both equal to 
  $\tilde{c}_{k,\varphi}^{k+1} Z^{\tilde{\alpha}}_\hh Z^{k}_3 \partial_{zz} p$.
  Using \eqref{pre}, we rewrite this term as
  \begin{align}
   \tilde{c}_{k,\varphi}^{k+1} Z^{\tilde{\alpha}}_\hh Z^{k}_3 \partial_{zz} p
   =
   -\tilde{c}_{k,\varphi}^{k+1} Z^{\tilde{\alpha}}_\hh Z^{k}_3 \Delta_\hh p
   -\tilde{c}_{k,\varphi}^{k+1} Z^{\tilde{\alpha}}_\hh Z^{k}_3 (\partial_i u_j \partial_j u_i)
   .\llabel{EQ97}
  \end{align}
  Since $Z^{\tilde{\alpha}}_\hh Z^{k}_3 (\partial_i u_j \partial_j u_i)$
  is of lower order, it may be estimated as in \eqref{EQ91a} and~\eqref{EQ91b}.
  For $Z^{\tilde{\alpha}}_\hh Z^{k}_3 \Delta_\hh p$, observe that
  \begin{align}
   \Vert Z^{\tilde{\alpha}}_\hh Z^{k}_3 \Delta_\hh p\Vert_{L^2}
   \lec
   \Vert D^2 Z^{\tilde{\alpha}}_\hh Z^{k}_3 p\Vert_{L^2}
   \lec 
   \Vert u\Vert_{4}(\Vert u\Vert_{2,\infty}+\Vert \eta\Vert_{L^\infty}+1)
   +
   \nu \Vert \nabla \nabla_\hh Z^{\tilde{\alpha}}_\hh u\Vert_{L^2}
   ,\llabel{EQ98}
  \end{align}
  thanks to the inductive hypothesis~\eqref{EQ.pre.k-1}.
  Repeating these estimates for the remaining terms in $I_1, I_2$, and $I_3$,
  we conclude that
  \begin{align}
   \Vert Z^\alpha \Delta p -\Delta Z^\alpha p\Vert_{L^2}
   \lec \Vert u\Vert_{4}(\Vert u\Vert_{2,\infty}+\Vert \eta\Vert_{L^\infty}+1)
   +
   \nu \Vert \nabla u\Vert_{4}
   .\label{EQ99}
  \end{align}
  Therefore, by \eqref{EQ93}, \eqref{EQ94}, and \eqref{EQ99},
  we obtain \eqref{EQ.pre.k},
  concluding the proof of Proposition~\ref{P.Pre}.
 \end{proof}
 
 \startnewsection{$L^\infty$ estimates}{sec.inf}
 
 In this section, we estimate $\Vert \eta\Vert_{L^\infty}$ and~$\Vert u\Vert_{2,\infty}$.
 
 \cole
 \begin{Proposition}
  \label{P.Inf}
  Let $u$ be a smooth solution of \eqref{NSE0}--\eqref{hnavierbdry}.
  Then we have the following statements.
  \begin{itemize}
   \item[i.] If $\mu \in \mathbb{R}$, we have
   \begin{align}
    \begin{split}
     \Vert u(t)\Vert_{2,\infty}^2+\Vert \eta(t)\Vert_{L^\infty}^2
     \lec& 
     \Vert u_0\Vert_{2,\infty}^2+\Vert \eta_0\Vert_{L^\infty}^2
     + \int_0^t (\Vert u\Vert_{2,\infty}+\Vert \eta\Vert_{L^\infty}+1)^3
     +\Vert u\Vert_{2,\infty}\Vert D^2 p\Vert_{3}\,ds
     \\&
     +\nu\int_0^t (\Vert \nabla \eta\Vert_{2}^2+\Vert \nabla u\Vert_{4}^2)\,ds,
     \label{EQ.Inf1}
    \end{split}
   \end{align}
   where the implicit constant is independent of~$\nu$.
   \item[ii.] If $\mu \ge 0$, we have
   \begin{align}
    \begin{split}
     \Vert u(t)\Vert_{2,\infty}^2+\Vert \eta(t)\Vert_{L^\infty}^2
     \lec& 
     \Vert u_0\Vert_{2,\infty}^2+\Vert \eta_0\Vert_{L^\infty}^2
     + \int_0^t (\Vert u\Vert_{2,\infty}+\Vert \eta\Vert_{L^\infty}+1)^3
     +\Vert u\Vert_{2,\infty}\Vert D^2 p\Vert_{3}\,ds
     \\&
     +\nu\int_0^t (\Vert \nabla \eta\Vert_{1}^2+\Vert \nabla u\Vert_{4}^2)\,ds
     ,
     \label{EQ.Inf2}
    \end{split}
   \end{align}
   where the implicit constant is independent of~$\nu$.
  \end{itemize} 
 \end{Proposition}
 \colb
 
 \begin{proof}[Proof of Proposition~\ref{P.Inf}]
  
  First, the maximum principle for \eqref{EQ.eta} implies
  \begin{align}
   \frac{d}{dt}\Vert \eta\Vert_{L^\infty}
   \lec
   \Vert \omega\Vert_{L^\infty}\Vert \nabla u\Vert_{L^\infty}
   + \Vert \nabla_\hh p\Vert_{L^\infty}
   \lec
   (\Vert \eta\Vert_{L^\infty} + \Vert u\Vert_{1,\infty})^2
   + \Vert p\Vert_{1,\infty},
   \label{EQ101}
  \end{align}
  where we have used \eqref{EQ62}, \eqref{EQ07}, and \eqref{EQ08}
  in the last inequality. 
  Note that the inequality independent of the sign of~$\mu$.
  Next, to treat the term $\Vert u\Vert_{2,\infty}$,
  we only consider the highest order derivative and note that lower order derivatives
  are bounded using the same methods.
  Now let 
  $Z^\alpha = Z^{\tilde{\alpha}}_\hh Z^k_3$ be a conormal
  derivative of order two, i.e., $|\tilde{\alpha}| + k =2$.
  We apply $Z^\alpha$ to \eqref{NSE0} and write
  \begin{align}
   Z^\alpha u_t
   - \nu \Delta Z^\alpha u
   + u\cdot \nabla Z^\alpha u
   = 
   u\cdot \nabla Z^\alpha u - Z^\alpha(u\cdot \nabla u) 
   - Z^\alpha \nabla p
   + \nu Z^\alpha \Delta u
   - \nu \Delta Z^\alpha u
   .
   \label{EQ102}
  \end{align}
  We complete the proof of Proposition~\ref{P.Inf} in two steps. 
  First, we take $p\ge4$, test the above equation with 
  $Z^\alpha u |Z^\alpha u|^{p-2}$, and then pass to the
  limit as~$p \to \infty$. 
  In the second step, we estimate the terms that emerge
  in the limit.\\
  
  \noindent\texttt{Step~1.} Test \eqref{EQ102} with $Z^\alpha u |Z^\alpha u|^{p-2}$.\\
  
  Upon integrating against $Z^\alpha u |Z^\alpha u|^{p-2}$, 
  the left-hand side of \eqref{EQ102} yields  
  \begin{align}
   \frac{1}{p}\frac{d}{dt}\Vert Z^\alpha u\Vert_{L^p}^p
   +\nu\int\left(
   \frac{(p-2)}{4} \big| \nabla |Z^\alpha u|^2 \big|^2 |Z^\alpha u|^{p-4}
   +|\nabla Z^\alpha u|^2 |Z^\alpha u|^{p-2}\right) \,dx
   + 2\mu \nu \delta_{0k} \Vert Z^\alpha u_\hh\Vert_{L^p(\partial \Omega)}^p
   ,\label{EQ103}
  \end{align}
  since the boundary term appears only for~$k=0$.
  When $\mu$ is non-negative, the last term in \eqref{EQ103} is coercive.
  This is the only instance where the sign of $\mu$ affects 
  the regularity assumption on~$\eta_0$. 
  In other words, when $\mu \ge 0$, we are able to close the estimates using $\eta_0 \in H^1_\cco(\Omega)$.
  On the other hand,
  when $\mu < 0$, we use the trace theorem and interpolation
  to obtain
  \begin{align}
   \nu\Vert Z^\alpha u_\hh\Vert_{L^p(\partial \Omega)}^p
   \lec
   \nu\Vert Z^\alpha u_\hh\Vert_{W^{\frac{1}{p},p}}^p
   \lec
   \nu\Vert Z^\alpha u_\hh\Vert_{L^p}^\frac{(p+4)p}{p+6} 
   \Vert \nabla Z^{\alpha} u_\hh\Vert_{L^6}^\frac{2p}{6+p}
   .\label{EQ145}
  \end{align}
  Recalling that $|\alpha|=2$, it is evident from  
  $\Vert \nabla Z^{\alpha} u_\hh\Vert_{L^6}$ that we need 
  to control at least two conormal derivatives of~$\eta$.
  Proceding with the right hand side of \eqref{EQ102},
  the first three terms give
  \begin{align}
   (u\cdot \nabla Z^\alpha u - Z^\alpha(u\cdot \nabla u) 
   - Z^\alpha \nabla p,
   Z^\alpha u |Z^\alpha u|^{p-2})
   \lec 
   \Vert u\cdot \nabla Z^\alpha u - Z^\alpha(u\cdot \nabla u) 
   - Z^\alpha \nabla p\Vert_{L^p}
   \Vert Z^\alpha u\Vert_{L^p}^{p-1} 
   \label{EQ104},
  \end{align}
  which leaves us with the commutator term for the Laplacian,
  for which we have
  \begin{equation}
   \label{EQ105}
   \nu Z^\alpha \Delta u
   - \nu \Delta Z^\alpha u =
   \begin{cases}
    0, &k = 0 \\
    -\nu 
    (2\varphi' \partial_{zz} Z_\hh u  + \varphi'' \partial_z Z_\hh u), &k=1  \\
     -\nu 
    (4\varphi' \partial_{zz} Z_3 u
    -4 (\varphi')^2 \partial_{zz} u
    +4\varphi'' \partial_z Z_3 u
    -5 \varphi' \varphi'' \partial_z u
    +\varphi''' Z_3 u), &k=2
    .
    \end{cases}
  \end{equation}
  Note that when $k\ge 1$ the leading order term is of the form
  $-\nu \varphi' \partial_{zz} Z u$, and to treat this term
  we use
  \begin{align}
   \partial_{zz} f
   = \frac{1}{\varphi} Z_3 \partial_z f
   = \frac{1}{\varphi} \partial_z Z_3 f 
   -\frac{\varphi'}{\varphi}\partial_z f
   .\label{EQ106} 
  \end{align}
  Upon letting $f = Zu$, we obtain
  \begin{align}
   \begin{split}
    -\nu \int \varphi' \partial_{zz} Zu Z_3 Z u |Z^\alpha u|^{p-2}\,dx
    =& -\nu 
    \int \frac{\varphi'}{\varphi} \partial_z Z_3 Z u Z_3 Z u |Z^\alpha u|^{p-2}\,dx
    \\&+\nu
    \int \frac{(\varphi')^2}{\varphi} \partial_z Zu Z_3 Zu |Z^\alpha u|^{p-2} \,dx
    = I_1 + I_2.
    \label{EQ107}  
   \end{split}
  \end{align}
  We point out that $I_1 \le 0$, and to show this, we write
  \begin{align}
   \begin{split}
    I_1&=
    -\frac{\nu}{p}
    \int \frac{\varphi'}{\varphi} \partial_z |Z^\alpha u|^p\,dx
    =\frac{\nu}{p}
    \int \frac{\varphi''\varphi - (\varphi')^2}{\varphi^2} |Z^\alpha u|^p\,dx
    \\&=\frac{\nu}{p}
    \int (\varphi''\varphi - (\varphi')^2) |\partial_z Z u|^2 |Z^\alpha u|^{p-2}\,dx
    \le 0,
    \label{EQ108}
   \end{split}
  \end{align} 
  where the last inequality follows from
  \begin{align}
   \varphi''(z)\varphi(z) - (\varphi'(z))^2 < 0 \comma z>0
   .\llabel{EQ109}
  \end{align}
  We proceed with $I_2$ by writing
  \begin{align}
   I_2 = 
   \nu \int (\varphi')^2 |\partial_z Z u|^2 |Z^\alpha u|^{p-2}\,dx
   \lec 
   \nu \Vert \partial_z Z u\Vert_{L^p}^2 \Vert Z^\alpha u\Vert_{L^p}^{p-2}
   .\label{EQ110}
  \end{align}
  Going back to \eqref{EQ105}, the next term we treat is~$\partial_{zz} u$.
  Upon applying \eqref{EQ106} for $f = u$, it follows that
  \begin{align}
   \begin{split}
    \nu \int (\varphi') \partial_{zz} u Z_3^2 u |Z_3^2 u|^{p-2}\,dx 
    &=\nu \int \frac{(\varphi')^2}{\varphi} \partial_z Z_3 u Z_3^2 u |Z_3^2 u|^{p-2}\,dx 
    -\nu \int \frac{(\varphi')^3}{\varphi} \partial_z u \partial_z Z_3 u |Z_3^2 u|^{p-2}\,dx 
    \\&=\nu \int (\varphi')^2 |\partial_z Z_3 u|^2 |Z_3^2 u|^{p-2}\,dx 
    -\nu \int (\varphi')^3 \partial_z u \partial_z Z_3 u |Z_3^2 u|^{p-2}\,dx 
    \\&\lec
    \nu 
    (\Vert \partial_z Z_3 u\Vert_{L^p}^2 
    +\Vert \partial_z u\Vert_{L^p}
    \Vert \partial_z Z_3 u\Vert_{L^p}) 
    \Vert Z^\alpha u\Vert_{L^p}^{p-2}
    .\label{EQ111}
   \end{split}
  \end{align}
We remark that the remaining terms in \eqref{EQ105}
are of lower order, and we estimate them using H\"older's inequality.
Therefore, collecting \eqref{EQ105}, \eqref{EQ107}, \eqref{EQ108},
\eqref{EQ110}, and \eqref{EQ111}, it follows that
  \begin{align}
   \begin{split}
    \nu\int  (Z^\alpha \Delta u- \nu \Delta Z^\alpha u)
    Z^\alpha u |Z^\alpha u|^2\,dx
    \lec I_1&+
    \nu\Vert Z^\alpha u\Vert_{L^p}^{p-2} 
    (\Vert \partial_z Z u\Vert_{L^p}^2 +\Vert \partial_z u\Vert_{L^p}\Vert \partial_z Z_3 u\Vert_{L^p})
    \\&+\nu \Vert Z^\alpha u\Vert_{L^p}^{p-1}
    (\Vert \partial_z Z u\Vert_{L^p} + \Vert \partial_z u\Vert_{L^p} + \Vert Z_3 u\Vert_{L^p})
    .\label{EQ112}
   \end{split}
  \end{align}
  Now, thanks to \eqref{EQ102}--\eqref{EQ104}, \eqref{EQ112} and since $I_1 < 0$
  we arrive at
  \begin{align}
   \begin{split}
    \frac{1}{p}\frac{d}{dt}\Vert Z^\alpha u\Vert_{L^p}^p
    \lec&
    \Vert Z^\alpha u\Vert_{L^p}^{p-1}
    (\Vert u\cdot \nabla Z^\alpha u - Z^\alpha(u\cdot \nabla u) 
    - Z^\alpha \nabla p\Vert_{L^p})
    \\&
    +\nu \Vert Z^\alpha u\Vert_{L^p}^{p-1}
    (\Vert \partial_z Z u\Vert_{L^p} + \Vert \partial_z u\Vert_{L^p} + \Vert Z_3 u\Vert_{L^p})
    \\&
    +\nu\Vert Z^\alpha u\Vert_{L^p}^{p-2} 
    (\Vert \partial_z Z u\Vert_{L^p}^2 +\Vert \partial_z u\Vert_{L^p}\Vert \partial_z Z_3 u\Vert_{L^p})
    \\&
    +\nu\delta_{k0}\Vert Z^\alpha u_\hh\Vert_{L^p}^\frac{(p+4)p}{p+6} 
    \Vert \nabla Z^{\tilde{\alpha}}_\hh u_\hh\Vert_{L^6}^\frac{2p}{6+p}    
    ,\llabel{EQ113}
   \end{split}
  \end{align}
  from where, canceling $\Vert Z^\alpha u\Vert_{L^p}^{p-2}$,
  letting $p\to \infty$, and integrating in time, we obtain
  \begin{align}
   \begin{split}
    \Vert Z^\alpha u(t)\Vert_{L^\infty}^2
    \lec&
    \Vert u_0\Vert_{2,\infty}+\int_0^t 
    \Vert u\Vert_{2,\infty}
    (\Vert u\cdot \nabla Z^\alpha u - Z^\alpha(u\cdot \nabla u)\Vert_{L^\infty})
    +\Vert Z^\alpha \nabla p\Vert_{L^\infty}+ \Vert u\Vert_{W^{1,\infty}})
    \\&+\nu\Vert u\Vert_{2,\infty}\Vert \partial_z u\Vert_{1,\infty} 
    +\nu 
    (\Vert \partial_z u\Vert_{1,\infty}^2 
    +\Vert u\Vert_{W^{1,\infty}}\Vert \partial_z u\Vert_{1,\infty}
    +\delta_{k0}\Vert \nabla Z^{\tilde{\alpha}}_\hh u_\hh \Vert_{L^6}^2)
    \,ds         
    .\label{EQ114}
   \end{split}
  \end{align}
  \\
  
  \noindent\texttt{Step~2.} We estimate the terms on the right hand side of~\eqref{EQ114}.\\
  
  We start with $\nu\Vert \partial_z u\Vert_{1,\infty}^2$. Utilizing  
  \eqref{EQ.emb}, for $(s_1,s_2) = (0,3)$, it follows that
  \begin{align}
   \nu\int_0^t \Vert \partial_z u(s)\Vert_{1,\infty}^2\,ds
   \lec \nu\int_0^t (\Vert \nabla \eta(s)\Vert_{1}^2 + \Vert \nabla u(s)\Vert_{4}^2)\,ds
   .\label{EQ115}
  \end{align}
  Next, for the term involving the horizontal velocity,
  recalling that it is present only when $\mu <0$, we write
  \begin{align}
   \nu\int_0^t \Vert \nabla Z^{\tilde{\alpha}}_\hh u_\hh (s)\Vert_{L^6}^2\,ds
   \lec 
   \nu \int_0^t(\Vert \nabla \eta\Vert_{2}^2
   +\Vert \nabla u_\hh(s)\Vert_{4}^2) \,ds
   ,
   \label{EQ146}
  \end{align}
  where we have used the embedding $H^1 \subset L^6$.
  Next, we consider the commutator term by writing
  \begin{align}
   u\cdot \nabla Z^\alpha u - Z^\alpha(u\cdot \nabla u)
   = -(1-\delta_{k0})\sum_{j=0}^{k-1}
   \tilde{c}^k_{j,\varphi} \frac{u_3}{\varphi} Z_3 Z^{\tilde{\alpha}}_\hh Z^j_3 u  
   -\sum_{1\le |\beta|\le |\alpha|}
   {\alpha \choose \beta}
   Z^\beta u \cdot Z^{\alpha-\beta} \nabla u
   =J_1 + J_2
   ,\label{EQ116}     
  \end{align}
  from where it follows that $J_1$ satisfies
  \begin{align}
   \Vert J_1\Vert_{L^\infty} \lec \Vert u\Vert_{1,\infty}\Vert u\Vert_{2,\infty}
   ,\label{EQ117} 
  \end{align}
  while for $J_2$ we have
  \begin{align}
   J_2 =
   \sum_{1\le |\tilde{\beta}| \le |\alpha|}
   {\alpha \choose \beta} 
   (Z^{\beta} u_\hh 
   \cdot 
   \nabla_\hh Z^{\alpha-\beta} u
   + 
   Z^{\beta} u_3 
   Z^{\alpha-\beta} \partial_z u)
   = J_{21} + J_{22}
   \lec \Vert u\Vert_{1,\infty}\Vert u\Vert_{2,\infty} + J_{22}.
   \label{EQ118}
  \end{align}
  Upon introducing lower order terms when necessary, we obtain
  \begin{equation}
  	\label{EQ119}
  	\Vert Z^{\beta} u_3  Z^{\alpha - \beta} \partial_z u\Vert_{L^\infty}
  	\lec
  	\begin{cases}
  		\left\Vert Z^{\beta} \frac{u_3}{\varphi}\right\Vert_{\infty} \Vert Z^{\alpha - \beta} Z_3 u\Vert_{L^\infty}
  		\lec \Vert u\Vert_{2,\infty}^2, &|\beta|=1
  		 \\
  		\Vert Z^{\beta} u_3\Vert_{L^\infty}\Vert Z^{\alpha - \beta} \partial_z u\Vert_{L^\infty}
  		\lec \Vert u\Vert_{2,\infty}(\Vert \eta\Vert_{L^\infty}+\Vert u\Vert_{1,\infty})
  		, &|\beta|=2
  		,
  	\end{cases}
  \end{equation}
from where, by collecting \eqref{EQ116}--\eqref{EQ119}
  and integrating in time, it follows that
  \begin{align}
   \int_0^t \Vert u(s)\cdot \nabla Z^\alpha u(s) - Z^\alpha(u(s)\cdot \nabla u(s))\Vert_{L^\infty} 
   \Vert u\Vert_{2,\infty}\,ds
   \lec \int_0^t \Vert u\Vert_{2,\infty}^2(\Vert \eta\Vert_{L^\infty}+\Vert u\Vert_{2,\infty})\,ds
   .\label{EQ120}
  \end{align}
  Going back to the right hand side of \eqref{EQ114}, 
  it only remains to estimate the pressure term.
  Another application of \eqref{EQ.emb} yields
  \begin{align}
   \Vert \nabla p\Vert_{2,\infty}
   \lec
   \Vert \partial_z p\Vert_{2,\infty}+\Vert p\Vert_{3,\infty}
   \lec
   \Vert D^2 p\Vert_{3}^\frac{1}{2}\Vert \nabla p\Vert_{4}^\frac{1}{2}+\Vert D^2 p\Vert_{3}
   \lec
   \Vert D^2 p\Vert_{3}
   .\label{EQ121}
  \end{align}
  
  To conclude, we multiply \eqref{EQ101} by $\Vert \eta\Vert_{L^\infty}$ and integrate in time.
  Then, combining with \eqref{EQ114} and considering
  \eqref{EQ115}, \eqref{EQ146}, \eqref{EQ120}, and \eqref{EQ121}
  we obtain~\eqref{EQ.Inf1}. When $\mu \ge 0$, we omit \eqref{EQ146} and arrive at~\eqref{EQ.Inf2}.
 \end{proof}

 \startnewsection{Concluding the a~priori estimates and the proofs of Theorems~\ref{T01} and~\ref{T03}(i)}{sec.apri}
 
 Note that combining \eqref{EQ.Con} and \eqref{EQ.Nor} for $m=2$ gives
 \begin{align}
  \begin{split}
   \Vert u(t)\Vert_{5}^2+\Vert \eta(t)\Vert_{2}^2
   &+c_0 \nu\int_0^t 
   (\Vert \nabla u(s)\Vert_{5}^2 + \Vert \nabla \eta(s)\Vert_{2}^2)\,ds
   \\
   &\lec \Vert u_0\Vert_{5}^2+\Vert \eta_0\Vert_{2}^2 
   + \int_0^t N^3\,ds
   + \int_0^t (\Vert \nabla p\Vert_{3}\Vert u\Vert_{4} + \Vert \nabla p\Vert_{2}\Vert \eta\Vert_{2})\,ds
   ,\llabel{EQ147}          
  \end{split}
 \end{align}
 recalling that $N$ is as in~\eqref{ap1}.
 Now, using the pressure estimates \eqref{EQ.Pre}, we have
 \begin{align}
  \begin{split}
   \Vert u(t)\Vert_{4}^2+\Vert \eta(t)\Vert_{2}^2
   &
   +c_0 \nu\int_0^t 
   (\Vert \nabla u(s)\Vert_{4}^2 + \Vert \nabla \eta(s)\Vert_{2}^2)\,ds
   \\
   &
   \lec N^2(0) + \int_0^t N^3\,ds
   + \int_0^t \nu\Vert \nabla u\Vert_{4}(\Vert u\Vert_{4}+\Vert \eta\Vert_{2})\,ds
   .\llabel{EQ25}          
  \end{split}
 \end{align}
 Therefore, upon absorbing the factors of $\nu\Vert \nabla u\Vert_{4}$, we obtain
 \begin{align}
  \Vert u(t)\Vert_{5}^2+\Vert \eta(t)\Vert_{2}^2
  +c_0 \nu\int_0^t 
  (\Vert \nabla u(s)\Vert_{5}^2 + \Vert \nabla \eta(s)\Vert_{2}^2)\,ds
  \lec N^2(0) + \int_0^t N^3\,ds
  .\label{EQ26}          
 \end{align}
 Now, we use \eqref{EQ26} in \eqref{EQ.Inf1} so that we have
 \begin{align}
  \Vert u\Vert_{2,\infty}^2+\Vert \eta\Vert_{L^\infty}^2
  \lec
  N^2(0)+\int_0^t N^3 \,ds
  +\int_0^t \Vert u\Vert_{2,\infty}\Vert D^2 p\Vert_{3}\,ds
  .\llabel{EQ130}
 \end{align}
 We may again use the pressure estimates \eqref{EQ.Pre}
 and absorb the emerging factors of~$\nu\Vert \nabla u\Vert_{4}$.
 Then, we combine the resulting inequality with \eqref{EQ26} which yields
 \eqref{ap1} and~\eqref{ap2}. To establish \eqref{ap3} and \eqref{ap4},
 we repeat these steps using \eqref{EQ.Con}, \eqref{EQ.Nor} for $m=1$, \eqref{EQ.Pre},
 and~\eqref{EQ.Inf2}. This concludes the proof of Proposition~\ref{P.Ap}.
 
 Now, having established the a~priori estimates,
 we may prove Theorems~\ref{T01} and~\ref{T03}(i) simultaneously. 
 First, we may smoothen the initial
 data and use well-posedness results to obtain solutions
 to which Proposition~\ref{P.Ap} applies.
 Then, we can pass to the limit in \eqref{NSE0}--\eqref{hnavierbdry}
 showing that the sequence of solutions is Cauchy in $L^\infty(0,T;L^2(\Omega))$.
 This step does not require $\nabla u \in H^2_\cco(\Omega)$ and hence it applies in both cases, i.e.,
 when $\nabla u \in H^1_\cco(\Omega)$ and $\nabla u \in H^1_\cco(\Omega)$.
 Moreover, for each $\nu$ we have 
 a unique solution thanks to the Lipschitz continuity
 of the initial data.
 Finally, employing the a~priori estimates
 \eqref{ap1}--\eqref{ap4}, the lifespans of both solutions, 
 which may be different, 
 are independent of~$\nu$. 
 
 \startnewsection{Proofs of Theorems~\ref{T02} and~\ref{T03}(ii)}{sec.main}
 
 We first prove that, under weaker regularity and boundedness assumptions,
 the sequence of solutions $\{u^\nu\}_\nu$ is Cauchy in $L^\infty(0,T;L^2(\Omega))$.
 
 \cole
 \begin{Proposition}[Compactness of $\{u^\nu\}_\nu$]
\label{P01}
Let $u_0 \in H^1(\Omega) \cap W^{1,\infty}(\Omega)$ be
such that $\div u_0 = 0$, and $u_0 \cdot n= 0$ on~$\partial \Omega$.
  Assume that
  there exists a $T>0$ such that for all $\nu \in (0,1)$
  and $\mu \in \mathbb{R}$ fixed there exists a unique solution
  $u^\nu \in L^\infty(0,T;H^1(\Omega)\cap W^{1,\infty}(\Omega))$
  to \eqref{NSE0}--\eqref{hnavierbdry} on~$[0,T]$.
Moreover, assume that there exists a constant $M>0$ such that for all~$\nu \in (0,1)$
  \begin{align}
   \sup_{[0,T]}
   (\Vert u^\nu(t)\Vert_{H^1}^2
   +\Vert u^\nu(t)\Vert_{W^{1,\infty}}^2)
   +\nu
   \int_0^T \Vert D^2 u^\nu(s)\Vert_{L^2}^2\,ds
   \le M.
   \label{EQ.comp}
  \end{align}
  Then, $\{u^\nu\}_\nu$ is a Cauchy sequence in $L^\infty(0,T;L^2(\Omega))$
  with
  \begin{align}
   \sup_{[0,T]}\Vert u^{\nu_1}-u^{\nu_2}\Vert_{L^2}^2
   \lec \nu_1 + \nu_2,
   \label{cauchy}
  \end{align}
  where the implicit constant is independent of $\nu_1$ and~$\nu_2$.
 \end{Proposition}
 \colb
 
 \begin{proof}[Proof of Proposition~\ref{P01}]
  Let $\nu_1,\nu_2 \in (0,1)$, and
  denote by $(u^1,p^1)$ and $(u^2,p^2)$ two solutions
  to \eqref{NSE0}--\eqref{hnavierbdry}, 
  emanating from the same initial data, 
  with the parameters $\nu_1$ and $\nu_2$ respectively.
  Upon letting $U=u_1-u_2$,
  and recalling that $\mu \ge 0$ is fixed, 
  we observe that the difference of solutions $U$ satisfies
  \begin{align}
   U_t - (\nu_1-\nu_2)\Delta u^1 - \nu_2 \Delta U
   +U \cdot \nabla u^1 + u^2 \cdot \nabla U + \nabla P 
   =0, \text{ and } \div U =0,
   \label{diff}
  \end{align} 
  where $P = p^1 - p^2$.
  This equation is coupled with the boundary conditions
  \begin{align}
   U_3 = 0 \text{ and } \partial_z U_\hh = 2\mu U_\hh
   \comma (x,t) \in \{z=0\}\times (0,T)
   .\label{b.diff}
  \end{align}
  Testing \eqref{diff} with $U$ and using~\eqref{b.diff},
  we obtain
  \begin{align}
   \begin{split}
    \frac{1}{2}\frac{d}{dt}\Vert U\Vert_{L^2}^2
    +\nu_2 \Vert \nabla U\Vert_{L^2}^2
    =&
    (\nu_2-\nu_1) \left(
    \int \nabla u^1 \nabla U \,dx
    +2\mu \int_{\partial \Omega} u^1_\hh U_\hh \,d\sigma\right)
    \\&
    -2\mu \nu_2 \Vert U_h\Vert_{L^2(\partial \Omega)}^2 
    - \int U \cdot \nabla u^1 U \,dx
    .\label{EQ149}
   \end{split}
  \end{align}
  As a consequence of the uniform bounds, we may 
  allow the implicit constant to depend on the upper bound~$C$
  in \eqref{EQ.comp} for the rest of this section.
  With this convention, \eqref{EQ149} implies
  \begin{align}
   \frac{d}{dt}\Vert U\Vert_{L^2}^2
   \lec
   (\nu_1+ \nu_2) (\Vert U\Vert_{H^1}\Vert u^1\Vert_{H^1} + \Vert U\Vert_{H^1}^2)
   +\Vert U\Vert_{L^2}^2\Vert \nabla u^1\Vert_{L^\infty}
   \lec (\nu_1+ \nu_2) + \Vert U\Vert_{L^2}^2
   ,\llabel{EQ151}  
  \end{align}
  from where, applying the Gronwall inequality on $(0,T)$, we get 
  \begin{align}
   \sup_{[0,T]}\Vert U\Vert_{L^2}^2 \lec \nu_1 + \nu_2
   ,\llabel{EQ154}
  \end{align}
  establishing~\eqref{cauchy}.
 \end{proof}
 
 We note that Proposition~\ref{P01} implies that 
 the sequence of solutions constructed in 
 Theorems~\ref{T01} and~\ref{T03}(i) is Cauchy in 
 $L^\infty(0,T;L^2(\Omega)$. This is true since Proposition~\ref{P01}
 holds under weaker regularity and boundedness assumptions and there is 
 no restriction on the sign of~$\mu$. To show that \eqref{invlimit} holds, 
 we use that $\nabla u^\nu \in L^\infty$ is uniformly bounded and 
 interpolate by writing
 \begin{align}
  \Vert u^1 - u^2\Vert_{L^\infty}^2
  \lec
  \Vert \nabla (u^1 -  u^2)\Vert_{L^\infty}^\frac{6}{5}\Vert u^1-u^2\Vert_{L^2}^\frac{4}{5}
  \lec
  (\nu_1+\nu_2)^\frac{2}{5}
  ,\llabel{EQ135}
 \end{align}
 where $u^1$ and $u^2$ are as in the proof of Proposition~\ref{P01}.
 Finally, passing to a subsequence concludes the proofs of Theorems~\ref{T02}
 and~\ref{T03}(ii). The uniqueness follows from the Lipschitz continuity of the data.
 
 \startnewsection{Euler equations in Sobolev conormal spaces}{sec.e}
 
 In this section, we sketch the proof of Theorem~\ref{T04}. For some $T>0$,
 we consider a smooth solution for \eqref{euler}--\eqref{eulerb} on $\Omega \times [0,T]$.
 We first note that, upon repeating the estimates leading to Proposition~\ref{P.Con},
 and to Proposition~\ref{P.Pre} for $\nu = 0$, we obtain
 \begin{align}
   \Vert u(t)\Vert_{4}^2
   \lec
   \Vert u_0\Vert_{4}^2
   +\int_0^t
   \bigl(
   \Vert u\Vert_{4}^2 
   (\Vert \partial_z u\Vert_{W^{1,\infty}}+
   \Vert u\Vert_{2,\infty}+1
   )+\Vert u\Vert_{4}\Vert \nabla p\Vert_{3}
   \bigr)\,ds
   ,\label{EQ160}
 \end{align}
 and
 \begin{align}
  \Vert D^2 p\Vert_{3}+\Vert \nabla p\Vert_{3}
  \lec 
  \Vert u\Vert_{4}(\Vert u\Vert_{2,\infty}+\Vert \partial_z u\Vert_{L^\infty}+1)
  ,\label{EQ161}
 \end{align} 
 where we have considered $\partial_z u$ instead of~$\eta$.
 Next, to estimate the Lipschitz norm of $u$,  we consider the vorticity formulation
 \begin{align}
  \omega_t + u \cdot \nabla \omega = \omega \cdot \nabla u 
  ,\llabel{EQ162}
 \end{align}
 from where the maximum principle leads to
 \begin{align}
  \frac{d}{dt}\Vert \omega\Vert_{L^\infty}^2
   \lec
    \Vert \omega\Vert_{L^\infty}^2\Vert \nabla u\Vert_{L^\infty}
    .\llabel{EQ163}
 \end{align}
 Next, we repeat the remaining estimates in Section~\ref{sec.inf}, recalling that $\nu=0$,
 yielding a bound on~$\Vert u\Vert_{2,\infty}$. Namely, we arrive at
  \begin{align}
    \Vert u(t)\Vert_{2,\infty}^2+\Vert \omega(t)\Vert_{L^\infty}^2
   \lec
   \Vert u_0\Vert_{2,\infty}^2+\Vert \omega_0\Vert_{L^\infty}^2
   + \int_0^t
   \bigl(
   (\Vert u\Vert_{2,\infty}+\Vert \omega\Vert_{L^\infty}+1)^3
   +\Vert u\Vert_{2,\infty}\Vert D^2 p\Vert_{3}
   \bigr)
   \,ds
   .\label{EQ164}
   \end{align}
 Since the term $\nu\Delta u$ does not appear, we do not obtain any
 boundary terms, and thus
 the assumptions on the sign of the friction parameter $\mu$ are not relevant.
 Now, we note that
 \begin{align}
  \Vert \nabla u\Vert_{L^\infty}
   \lec
    \Vert \omega\Vert_{L^\infty}+\Vert u\Vert_{1,\infty}
    .\label{EQ165}
 \end{align}
 Upon combining, \eqref{EQ160}, \eqref{EQ161}, and \eqref{EQ164},
 we employ \eqref{EQ165} to conclude~\eqref{EQ.main3}. 
 
 Having established the a~priori estimates, we may construct the unique solution
 in the following manner. We first approximate our initial data with a sequence satisfying
 \eqref{ourassumption2} and employ Theorem~\ref{T03}(ii) to obtain a sequence of approximate solutions.
 Note that the a~priori estimate established above applies to this sequence for some~$T>0$.
 Next, we show that the same $T$, which is independent of the approximation, 
 is a lower bound on maximum time of existence of the approximate solutions. 
 This is possible since the a~priori bounds linearizes the estimate on the conormal derivative of $\nabla u$;
 see~\eqref{EQ.Nor}. Finally, upon showing that the sequence of approximate solutions is Cauchy in $L^\infty(0,T;L^2(\Omega))$ we may pass to the limit. A detailed construction 
 will be addressed elsewhere.

 \colb
 \section*{Acknowledgments}
 The authors were supported in part by the
 NSF grant DMS-2205493.
 The research was performed during the program ``Mathematical Problems in Fluid Dynamics, Part~2,'' held during the summer of 2023 by the Simons Laufer Mathematical Sciences Institute (SLMath), which is supported by the National Science Foundation (Grant No.~DMS-1928930).

\end{document}